%% file: renorm.tex
\documentclass[12pt]{amsart}
\usepackage[utf8]{inputenc}
\usepackage[english]{babel}
\usepackage{amsmath,amssymb,amsfonts,amsthm}
\usepackage{graphicx}
\usepackage{tikz}

 \usepackage[draft]{fixme}
\newcounter{thm}
\newtheorem{theorem}[thm]{Theorem}
\newtheorem*{theorem*}{Theorem}
\newtheorem{lemma}[thm]{Lemma}
\newtheorem{techlemma}[thm]{Technical Lemma}
\newtheorem{proposition}[thm]{Proposition}
\newtheorem*{conjecture*}{Conjecture}
\newtheorem{corollary}[thm]{Corollary}

\newtheorem{remark}[thm]{Remark}
\theoremstyle{definition}
\newtheorem{definition}[thm]{Definition}

\renewcommand{\Re}{\mathop{\mathrm{Re}}}

\renewcommand{\Im}{\mathop{\mathrm{Im}}}
\newcommand{\dist}{\operatorname{dist}}
\newcommand{\eps}{\varepsilon}
\newcommand{\bbC}{\mathbb C}
\newcommand{\CC}{\mathbb C}
\newcommand{\NN}{\mathbb N}
\newcommand{\RR}{\mathbb R}
\newcommand{\TT}{{\mathbb R/\mathbb Z}}
\newcommand{\ZZ}{\mathbb Z}
\newcommand{\bbR}{\mathbb R}
\newcommand{\bbQ}{\mathbb Q}

\newcommand{\bbD}{\mathbb D}

\newcommand{\bbZ}{\mathbb Z}
\newcommand{\bbA}{\mathbb A}
\newcommand{\WW}{\mathbb W}
\newcommand{\cB}{\mathcal B}
\newcommand{\cD}{\mathcal D}
\newcommand{\cR}{\mathcal R}
\newcommand{\cU}{\mathcal U}
\newcommand{\cV}{\mathcal V}
\newcommand{\cT}{\mathcal T}

\newcommand{\rot}{\rho}
\renewcommand{\dist}{\mathrm{dist}\,}

\renewcommand{\mod}{\operatorname{mod}}

\numberwithin{thm}{section}
\author{Nataliya Goncharuk, Michael Yampolsky}
\title[Analytic linearization]{Analytic linearization of conformal maps of the annulus}

\begin{document}
\begin{abstract}
We consider holomorphic  maps defined in an annulus around $\bbR/\bbZ$ in $\bbC/\bbZ$.
E.~Risler proved that in a generic analytic family of such maps $f_\zeta$ that contains a Brjuno rotation $f_0(z)=z+\alpha$, all maps that are conjugate to this rotation form a codimension-1 analytic submanifold near $f_0$.

 In this paper, we obtain the Risler's result as a  corollary of the following construction. We introduce a renormalization operator on the space of univalent maps in a neighborhood of $\bbR/\bbZ$. We prove that this operator is hyperbolic, with one unstable direction corresponding to translations. We further use a holomorphic motions argument and Yoccoz's theorem to show that its stable foliation consists of diffeomorphisms that are conjugate to rotations.
\end{abstract}

\maketitle

\input{intro}

\input{construction}

\input{operator}

\input{stable}
\input{conclusion}

\bibliographystyle{amsalpha}
\bibliography{biblio}

 \end{document}

%% file: intro.tex
\section{Introduction}

\subsection{Linearization of analytic diffeomorphisms close to rotations}
We denote $\{x\}$, $[x]$ the fractional and the integer parts of a real number $x$ respectively.
For $\alpha\in(0,1)$, let $$G(\alpha)=\left\{\frac{1}{\alpha} \right\}$$
denote the Gauss map. We set
$$\alpha_{-1}\equiv 1,\;\alpha_0\equiv\alpha,\ldots,\alpha_n=G(\alpha_{n-1}),\ldots;$$
this sequence is infinite if and only if $\alpha\notin \bbQ$, otherwise, we will end it at the last non-zero term.
Note that numbers $a_n=[1/\alpha_{n}]$, $n\geq 0$ are the coefficients of a continued fraction expansion of $\alpha$ with positive terms (which is unique  if $\alpha\notin \bbQ$).
To save space, we will abbreviate this continued fraction as
$$\alpha=[a_0,a_1,\ldots].$$
As usual, $p_n/q_n$ will denote the $n$-th convergent of the continued fraction of $\alpha$:
$$\frac{p_n}{q_n}=[a_0,\ldots,a_{n-1}].$$
An irrational number $\alpha\in(0,1)$ is a Brjuno number if the following sum converges:
\begin{equation}
  \label{BY-function}
  \Phi(\alpha)=\sum_{n\geq 0}\alpha_{-1}\alpha_0\cdots\alpha_{n-1}\log\frac{1}{\alpha_n}.
  \end{equation}
This sum is known as Yoccoz-Brjuno function \cite{Yoc}. Its convergence is equivalent to that of the original Brjuno function \cite{Bru}:
\begin{equation}
  \label{B-function}
  \Phi_0(\alpha)=\sum_{n\geq 0}\frac{\log{q_{n+1}}}{q_n},
  \end{equation}
  and the difference $|\Phi(\alpha)-\Phi_0(\alpha)|$ is bounded by a universal constant.

We will denote the collection of Brjuno numbers by $\cB\subset(0,1)$.

Let $\mathcal B_C = \{\alpha\in \mathcal B \mid \Phi(\alpha)<C\}$.
  For a positive decreasing sequence $\{s_k\}$ that tends to zero, let $\mathcal B_{\{s_k\}}$ be the set of Brjuno numbers such that for all $k$, $$\Phi^k(\alpha) = \sum_{n\geq k}\alpha_{-1}\alpha_0\cdots\alpha_{n-1}\log\frac{1}{\alpha_n}<s_k.$$ Then $\mathcal B$ is a union of $\mathcal B_{\{s_k\}}$ over all decreasing sequences $\{s_k\}$ that tend to zero.

We will denote by $\rho(f)$ the rotation number of  a circle homeomorphism $f:\bbR/\bbZ\to\bbR/\bbZ$.
We let
$$R_\alpha(z)\equiv (z+\alpha)\mod \bbZ,$$
where $\alpha\in\bbC$.
Let  $$\Pi_{\eps}:=\{z\in \bbC/\bbZ \mid |\Im z|<\eps\}$$ be  an annulus in $\bbC/\bbZ$ around $\bbR/\bbZ$ of width $2\eps$.

The following theorem was proved by E.~Risler in \cite{Risler}:
  \begin{theorem}[{\bf Risler}]
  \label{th-Ris}
  Let $\alpha\in\cB$. Let $f_\mu$, $\mu\in \bbD^n_r(0)\subset \bbC^n$ be a generic analytic family of univalent maps
  $$f_\mu:\Pi_\eps\to\CC/\ZZ$$
  satisfying the condition $f_0=R_\alpha$. Then there exists $\kappa>0$ and an analytic submanifold $M\ni 0$ of $\bbD_r^n(0)$ whose codimension is equal to $1$ (i.e. $\dim M=n-1$) such that for all $\mu\in \bbD^n_\kappa(0)\cap M$ there exists an analytic homeomorphism
  $h_\mu:\Pi_{\frac{\eps}{2}}\to\CC/\ZZ$ with the property
  $$h_\mu\circ f_\mu\circ h_\mu^{-1}=R_\alpha.$$
 Moreover, if $f_\mu$ is analytically conjugate to $R_\alpha$ in a substrip of $\Pi_{\eps/2}$ and $\mu$ is sufficiently small, then $f_\mu \in M$.
  \end{theorem}
  The proof in \cite{Risler} is deeply motivated by the results of Yoccoz on analytic linearization of circle diffeomorphisms \cite{Yoccoz2002}. It  involves constructing a sequence of Yoccoz renormalizations of $f_\mu$; the renormalized maps are defined in progressively taller  cylinders,
  whose lifts to $\Pi_\eps$ are shown to contain a sub-annulus around $\RR/\ZZ$. This enables the construction of  an analytic chart that uniformizes $f_\mu$.
  Yoccoz \cite{Yoccoz2002} originally used this approach to get the following result:

  \begin{theorem}[{\bf Yoccoz}]
   \label{th-Yoc1}
   For every positive sequence  $\{s_k\}, \, s_k\to 0$, $\eps>0$ there exists $\kappa>0$ such that the following holds. Let $\alpha\in\cB_{\{s_k\}}$, and let $f$ be an analytic circle diffeomorphism with rotation number $\alpha$. Let $f$ extend conformally to $\Pi_\eps$, with
   $$|f(z)-R_\alpha(z)|<\kappa \quad \text{ for any }  z\in\Pi_\eps.$$
   Then there exists a conformal real-symmetric change of coordinates $h$ defined in $\Pi_{\frac{\eps}{2}}$ such that
   $$h\circ f\circ h^{-1}=R_\alpha.$$
  \end{theorem}
  In fact, Yoccoz also proved a quantitative version of this theorem, which shows more clearly the role of the Yoccoz-Brjuno function:
  \begin{theorem}[{\bf Yoccoz}]
    \label{th-Yoc2}
    There exists a universal constant $C_0>0$ such that the following holds.
   Let $f$ be an analytic diffeomorphism of the circle, with rotation number $\alpha=\rho(f)\in\cB$. Suppose that $f$ is an analytic and univalent function in $\Pi_\eps$ with
   $$\eps>\frac{1}{2\pi}\Phi(\alpha)+C_0.$$
   Then, there exists a conformal real-symmetric change of coordinates
   $$h:\Pi_{\eps'}\to\CC/\ZZ\text{ where }\eps'=\eps-\frac{1}{2\pi}\Phi(\alpha)-C_0$$ such that
   $$h\circ f\circ h^{-1}=R_\alpha.$$
   \end{theorem}

   \begin{remark}
     In \cite{Yoccoz2002}, Theorem \ref{th-Yoc1} is formulated in a different way. Namely, Yoccoz states this theorem for an arbitrary Brjuno number $\alpha$, with the number $\kappa$ depending on $\alpha$. His proof, however, shows that we may choose  $\kappa$ uniformly for all $\alpha\in\cB_{\{s_k\}}$, see Proposition \ref{prop-Yoc} below and the discussion after it. This is proved implicitly in \cite[Sec.4.3]{Yoccoz2002}.
    \end{remark}



 We note that, although the methods of the proof in \cite{Yoccoz2002} and \cite{Risler} are similar, it was by no means clear what the connection is between Yoccoz's theorems and
 the result of Risler.

\subsection{Main results}
Our principal motivation was to  re-prove Theorem \ref{th-Ris}  using a different approach. In fact, we prove the following, stronger statement.
Let  $\mathcal D_{\eps}$ be the affine Banach space formed by bounded analytic maps $$f\colon \Pi_\eps\mapsto\bbC/\bbZ$$ that are defined in $\Pi_\eps$, and extend continuously to the boundary,  with the $\sup$-norm in this annulus.

  \begin{theorem}
    \label{th-submanif}
    For every $\{s_k\}\to 0$ and $\eps>0$ there exists $\kappa>0$ such that the following holds.
  Let $\alpha\in\cB_{\{s_k\}}$.
  Then the set of maps $f\in \mathcal D_{\eps}$ such that $f$ is analytically conjugate to $R_\alpha$ in  $\Pi_{\eps/2}$ and whose distance  to $R_\alpha$ is bounded by $\kappa$ forms an embedded analytic submanifold of $\mathcal D_\eps$ at $R_\alpha$ of codimension 1.
\end{theorem}
  Theorem~\ref{th-submanif} clearly implies Risler's Theorem \ref{th-Ris}.

This result is a corollary of the properties of a {\it renormalization operator} which we construct in this paper. Before describing them,
let us give a few useful definitions.
Suppose, $\mathbf B$ is a complex Banach space whose elements are functions of a complex variable. Following the
notation of \cite{Ya3}, let us say that
the {\it real slice} of $\mathbf B$ is the real Banach space $\mathbf B^\RR$ consisting of the real-symmetric elements of $\mathbf B$.
If $\mathbf X$ is a Banach manifold modelled on $\mathbf B$ with the atlas $\{\Psi_\gamma\}$
we shall say that $\mathbf X$ is {\it real-symmetric} if $\Psi_{\gamma_1}\circ\Psi_{\gamma_2}^{-1}(U)\subset \mathbf B^\RR$ for any pair of indices $\gamma_1$, $\gamma_2$ and any open set $U\subset\mathbf B^\RR$, for which the above composition is defined. The {\it real slice of $\mathbf X$} is then defined as the real
Banach manifold $\mathbf X^\RR= \cup_\gamma\Psi_\gamma^{-1}(\mathbf B^\RR) \subset \mathbf X$
with an atlas $\{\Psi_\gamma\}$.
An operator $A$ defined on a {subset} $Y\subset\mathbf X$ is {\it real-symmetric} if $A(Y\cap\mathbf X^\RR)\subset \mathbf X^\RR$.


\begin{theorem}[{\bf Main Theorem}]
\label{th-main}
For sufficiently large $\eps$, there exists a renormalization operator $\mathcal R\colon \mathcal D_{\eps}\to \mathcal D_{\eps}$ with the following properties.
\begin{enumerate}
\item $\mathcal R$ is defined in a neighborhood of the set $$\mathcal T=\{R_\alpha \mid \alpha\in (\bbR/\bbZ)\setminus K\}$$ in $\mathcal D_{\eps}$ where $K\subset\bbQ/\bbZ$ is a countable set that accumulates only at rational numbers $p/q$ with $q\le 100$;
 \item $\mathcal R$ is a real-symmetric complex-analytic operator with compact differential at each $R_\alpha\in \mathcal T$;
 \item  \label{it-hyp} For each $C$, for  $\eps>c_1 C + c_2$ where $c_1, c_2$ are universal constants, the operator $\mathcal R$ is hyperbolic  with 1-dimensional unstable direction on the set
 $$
 \{R_\alpha \mid \alpha\in \mathcal B_{C}\}.
 $$
 At each point of this set, the following properties hold:
 \begin{itemize}
  \item  $\mathcal R$ has an unstable manifold with uniform expansion;
 \item   There exists a local codimension 1 analytic submanifold $\mathcal V_\alpha$ such that powers of $\mathcal R$ contract on it, with a uniform rate of contraction;

 \item $\mathcal V_\alpha$ only contains   diffeomorphisms $f$ that are analytically conjugate to $R_\alpha$:  $f = \xi R_\alpha \xi^{-1}$, $\xi(0)=0$, where $\xi$ is defined in $\Pi_{0.4\eps}$. Moreover, if $f$ is analytically conjugate to $R_\alpha$ in a substrip of $\Pi_{\eps/3}$ and sufficiently (depending on $\alpha$) close to $R_\alpha$, then  $f \in \cV_\alpha$.

 \end{itemize}

\end{enumerate}

\end{theorem}
 Theorem \ref{th-submanif} is a direct corollary of the above theorem for sufficiently large values of $\eps$. For small $\eps$, one needs to renormalize several times to increase the domain where the maps are defined, and then refer to Theorem \ref{th-main}. See \S~\ref{sec-conclusion} for more details.

The construction of the renormalization operators below is inspired by the cylinder renormalization transformation defined in \cite{Ya3}, and its adaptation in \cite{GorYa}. The
difference with Yoccoz's renormalization approach \cite{Yoccoz2002} is that we employ renormalization not only as a geometric tool, but as a dynamical system acting on a functional space. This allows us to study the dynamics of this operator, which is hyperbolic, with a one-dimensional unstable direction. This ensures the existence of the stable foliation of unit codimension from the general theory of hyperbolic dynamics.

The proof of hyperbolicity of the renormalization operator is surprisingly straightforward.
However, the existence of the conjugacy to $R_\alpha$ for $f$ in the stable leaf of $R_\alpha$ is, of course, not an automatic consequence of the general theory. It is here that we use Yoccoz's Theorem~\ref{th-Yoc1} -- and thus establish a direct connection between Yoccoz's and Risler's results.

Since the operator $\cR$ is real-symmetric, the same properties hold for its restriction to the real slice $\cD_\eps^\RR$, so, in particular:
\begin{corollary}
\label{cor-submanif1}
  Let $\alpha$ be a Brjuno number.
  For any positive $\eps$, the set of circle homeomorphisms $f\in \mathcal D^\RR_{\eps}$ such that $\rho(f)=\alpha$
  forms a local analytic submanifold
  of $\mathcal D^\RR_\eps$ at $R_\alpha$ of codimension 1.
\end{corollary}

\noindent
We can apply our results to specific families of maps which are transverse to the stable foliation of $\cR$.

For instance, consider the Arnold family
$$F_{\mu,a}(z)=z+\mu-\frac{a}{2\pi}\sin(2\pi z)\mod \bbZ,\;\mu\in[0,1),\; a\in[0,1).$$
    Recall that the Arnold's tongue of rotation number $\alpha$ is defined as
    $$\bbA_\alpha=\{(\mu,a)\;|\;\rho(F_{\mu,a})=\alpha\}.$$
\begin{corollary}
\label{cor-tongue}
For each $\{s_k\}\to 0$ there exists $\delta_{\{s_k\}}>0$, such that the set
$\{\bbA_\alpha,\;\alpha\in\cB_{\{s_k\}},a\in[0,\delta_{\{s_k\}})\}$ is a foliation by real-analytic curves over $a\in[0,\delta_{\{s_k\}})$.
\end{corollary}
Compare this with the discussion in the introduction of \cite{Risler}. A similar statement can be obtained for the complex Arnold family. See  \cite{FaGe} for a global result on the analytic parametrization of Arnold's tongues in terms of the moduli of Herman rings.

Another example is the family of Blaschke fractions
\begin{equation}
  \label{eq:blfr}
f_{\mu,a}(z)=\mu z^2\frac{az+1}{z+a},\;\mu\in \CC^*,a\in\CC.
  \end{equation}
This family has been studied in the literature as an example of Shishikura's quasiconformal surgery \cite{Shisurg} turning quadratic Siegel disks into Herman rings
(see  \cite{bfgh} and the discussion therein). In particular, an Arnold disk $\WW_\alpha\ni 0$ is defined in \cite{bfgh} as
the set of $(\mu,a)\in\CC^*\times \CC$ such that $f_{\mu,a}$ has a fixed Herman ring with rotation number $\alpha$ (we include the case $f_{\alpha,0}\equiv R_\alpha$).
Then we can derive:
\begin{corollary}
  \label{cor-disk}
  For each  $\{s_k\}\to 0$ there exists $r_{\{s_k\}}>0$ such that the set of Arnold disks $\{W_\alpha,\;\alpha\in\cB_{\{s_k\}}\}$ forms a foliation by complex analytic graphs over $|a|<r_{\{s_k\}}$.
\end{corollary}
Compare this to \cite[Theorem A, Theorem B]{bfgh}. For Blaschke products, Theorem A  shows that each Arnold disc is a codimension-1 analytic submanifold of $\bbC^2$, and provides a parametrization of this manifold in terms of the geometric properties of the Herman ring. Theorem B shows that locally, near rotations, the Arnold disc can be parametrized by $a$.  Corollary \ref{cor-disk} above provides a uniform estimate on the projection of this disc to $a$ for $a\in \mathcal B_{\{s_k\}}$.

As the reader will see in the next section, our definition of renormalization largely parallels that of Yoccoz \cite{Yoccoz2002}, and is based on a carefully selected conformal rescaling of a first return map of an analytic map which is close to a rotation. We have to be more careful with the choice of the rescaling, so that the renormalization transformation becomes a well-defined complex-analytic operator in a neighborhood of the rotations. There are multiple advantages to viewing renormalization as an operator in a functional space, which go beyond the streamlined proof of Risler's theorem. For instance, since the preprint version of this paper appeared, the second author was able to extend the hyperbolic action of the renormalization operator to dissipative maps of two  complex variables in \cite{KAM-yam}, and proved a result on persistence of Herman rings for small 2D perturbations of one-dimensional maps.

%% file: construction.tex
\section{Defining renormalization}
\label{sec-def}
\subsection{Family of renormalization transformations}
 Here we define a family of renormalization transformations $\mathcal R_{w,n}$; later we will fix the choice of $w$ and $n$ to define the renormalization operator $\mathcal R$.

Fix any $\eps\in \bbR^+$. Let $\{p_l/q_l\}$ be the sequence of the continued fraction convergents of $\alpha$, and let $n=q_{m}$ for some $m$. Consider a map $f\in \mathcal D_\eps$ that is close to a  rotation $R_\alpha$, $\alpha\in\bbR/\bbZ$, with $\alpha\notin \bbQ$  (so that, in particular, $f$ is univalent).
Fix $w\in \bbR^+$;  $w$ stands for width. Let $l=\{n\alpha\}$ and $L = f^n(0)-0$; note that $L\approx l$ if $f$ is close to $R_\alpha$.

Assume $l\cdot w<\eps$. Consider a segment $I = [-i w L, i w L] \subset \Pi_\eps$; this segment is close to vertical. Let $R$ be the curvilinear rectangle bounded by $I$, $f^n(I)$, and two straight segments joining their endpoints. If  $f\in \mathcal D_{\eps}$ is sufficiently close to a rotation, these four curves are simple and bound a domain in $\Pi_{\eps}$. We will assume that $R$ includes $I$, $f^n(I)$ and does not include the straight segments that join endpoints of $I, f^n(I)$;
then  $f^n(R)\cap R = f^n(I)$.

Consider a biholomorphic map $\Psi \colon R \to \bbC$, $\Psi(0)=0$, that conjugates $f^n$ to the shift by one and is defined on $R$. We will see in Proposition \ref{prop-Psi-estim}  below that $\Psi$ can be chosen so that it depends analytically on $f$ and is close to the linear expansion $z\mapsto \frac{z}{l}$ whenever $f$ is close to $R_\alpha$. The map $\Psi$ descends to the map $\widetilde \Psi\colon R/f^n \mapsto \bbC/\bbZ$.

Let $P$ be the first-return map to $R$ under the iterates of $f$.

\begin{definition}
\textbf{Renormalization transformation} of $f\in \mathcal D_{\eps}$ is
$$\mathcal R_{w, n} f \equiv \widetilde \Psi P \widetilde \Psi^{-1}.$$
\end{definition}

Since the first-return  map $P$ descends to the  continuous, analytic, and univalent map on a subdomain of the annulus $A=R/f^n$, the map $\Psi P \Psi^{-1}$ induces an analytic univalent map  on a subset of $\bbC/\bbZ$.
For $f$ close to the rotation $R_\alpha$, $\Psi$ is close to $z\mapsto z/l$, thus the domain of $\mathcal R_{w,n} f$ is close to the strip $\Pi_w$.

The following proposition describes the construction and properties of the chart $\Psi$.
\begin{proposition}
\label{prop-Psi-estim}
Let $\eps, w, n$ be fixed, $l \cdot w <\eps$.

There exists a neighborhood  $\mathcal U_\alpha \subset \mathcal D_\eps$ of $R_\alpha$ such that for each $f\in \mathcal U_\alpha\subset \mathcal D_\eps$, there exists a biholomorphic map $\Psi$ with the following properties.
\begin{enumerate}


\item \label{it-unif-domain} $\Psi$ conjugates $f^n$ to the unit translation $z\mapsto z+1$ and is defined on the quadrilateral $R$.

\item \label{it-real} If $f$ preserves $\bbR/\bbZ$, then so does $\Psi$.

\item \label{it-lin} If $f=R_\alpha$, then $\Psi(z)=z/l$. If $f$ is sufficiently close to $R_\alpha$ in $C(\Pi_{\eps})$, then $\Psi$ is close to $z\mapsto z/l$ in $C(R)$.

 \item \label{it-an-depend} Consider a family $\zeta\mapsto f_\zeta\in \mathcal D_\eps$, $\zeta\in \bbC^m$, that depends analytically on $\zeta$ with $f_0=R_\alpha$.
 Then $\Psi$ depends analytically on $\zeta$ for $\zeta$ in some neighborhood of zero.





\end{enumerate}

\end{proposition}
\begin{proof}

 $\;$ \\\noindent
  \textbf{Step 1. Construction of $\Psi$ via solutions of the Beltrami equation;  item \eqref{it-unif-domain}}

  \smallskip

\noindent
{\sl Step 1.1. Stretching the fundamental domain. }
\\
The map $\Psi$ will be constructed as $\Psi = \Theta (z/L)$ where $\Theta$ is close to identity. Note that $\Theta$ should be a holomorphic map that conjugates $g(z):=f^n(Lz)/L$ to the shift by $1$.
Clearly, if $f$ is close to $R_\alpha$, then $g$ is close to $z\mapsto z+1$.

\begin{figure}[h]
\begin{center}
  \begin{tikzpicture}
 \begin{scope}[rotate=10]
  \draw (0, -0.5 ) -- (0, 0.5);
 \draw plot [smooth] coordinates { (1, -0.6) (1.1, 0) (0.95, 0.25) (1, 0.6) };
\draw (0, 0.5)  --(1, 0.6);
\draw (0, -0.5)  --(1, -0.6);

\draw (0,0) node [left] {$0$} -- (1.1, 0) node [right]{$L$};

 \draw [-> ] plot [smooth] coordinates { (-0.1, 0.3) (-0.3, 0.4) (-0.2, 0.7) (0.5, 0.8)  (1.2, 0.7) (1.3, 0.4) (1.1, 0.3)};
 \draw  (0.5, 0.8) node [above]{$f^n$};

 \end{scope}

 \draw [->, line width = 1](0.5, -0.7) -- (0.5, -1.7);
\draw (0.5, -1) node [left]{$\cdot \frac 1L$};

  \begin{scope}[yshift = -80]
  \draw (0, -0.5 ) -- (0, 0.5);
 \draw plot [smooth] coordinates { (1, -0.6) (1.1, 0) (0.95, 0.25) (1, 0.6) };
\draw (0, 0.5)  --(1, 0.6);
\draw (0, -0.5)  --(1, -0.6);
 \draw [-> ] plot [smooth] coordinates { (-0.1, 0.3) (-0.3, 0.4) (-0.2, 0.7) (0.5, 0.8)  (1.2, 0.7) (1.3, 0.4) (1.1, 0.3)};
 \draw  (1.3, 0.4) node [above  right]{$g$};

\draw (0,0) node [left] {$0$} -- (1.1, 0) node [right]{$1$};
 \end{scope}

  \draw [->, line width = 1](4, -2.8) -- (2, -2.8);
  \draw (3, -2.8) node [above]{$H$};

  \begin{scope}[yshift = -80, xshift = 130]
  \draw (0, -0.5 ) -- (0, 0.5);
\draw (0, 0.5)  --(1, 0.5);
\draw (0, -0.5)  --(1, -0.5);
\draw (1, -0.5) -- (1, 0.5);
 \draw [-> ] plot [smooth] coordinates { (-0.1, 0.3) (-0.3, 0.4) (-0.2, 0.7) (0.5, 0.8)  (1.2, 0.7) (1.3, 0.4) (1.1, 0.3)};
 \draw  (1.3, 0.4) node [above  right]{$+1$};
 \draw (0.3, -0.2) node {$\mu$};
\draw (0,0) node [left] {$0$} -- (1, 0) node [right]{$1$};
 \end{scope}

   \draw [->,  line width = 1](5, -1.8) -- (5, -0.8);
  \draw (5, -1.3) node [left]{$F$};

  \draw (5, 0)  node {$\mathbb C$};

   \draw [->, line width = 1](2, 0) -- (4, 0);
\draw (3, 0) node [above]{$\Psi$};

\draw  [->, line width = 1] (1.5, -1.7) -- (4, -0.5);
\draw (2.5, -1.1) node [above ]{$\Theta$};

\end{tikzpicture}
\end{center}
\caption{The commutative diagram: the construction of $\Psi$.}\label{fig-cd}

\end{figure}
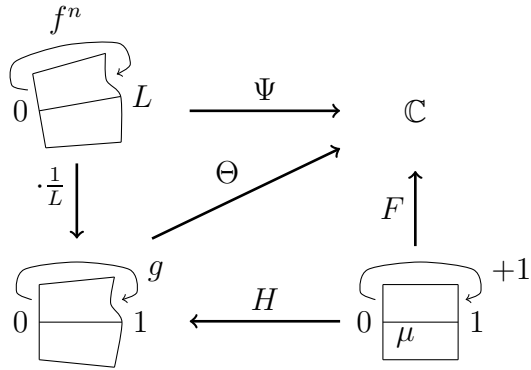

  \smallskip

\noindent
{\sl Step 1.2. Smooth straightening of the fundamental domain.}
\\
Put $Q = [0,1]\times[-i w, i w]$ and consider the following smooth map $H \colon Q \to \bbC$:
$$H(t+is) = (1-t)\cdot is   +  t\cdot g(is).$$
This map takes horizontal segments $[0+is, 1+is]$ to slanted segments $[is, g(is)]$, thus $H^{-1}$  conjugates $g$ to the horizontal shift $z\mapsto z+1$; also, $H(0)=0$.

Note that if $g$ is close to the shift by $1$ in $\Pi_{\eps/l}$ (which is the case if $f$ is close to $R_\alpha$ in $C(\Pi_\eps)$),   then $H$ is a well-defined $C^2$-smooth map inside $Q$, and is close to identity in $C^2(Q)$.
Indeed, differentiation with respect to $t,s$ yields
$$H'_t = -is + g(is), \qquad H'_s = i(1-t)+itg'(is), $$
$$
H''_{tt} = 0, \quad H''_{ts} = i(-1+g'(is)), \quad H''_{ss} = -tg''(is).
$$
When  $f$ is close to $R_\alpha$ in $\Pi_\eps$,  $g$ is close to the shift by $1$ in $\Pi_{\eps/l}$, thus is $C^2$-close to the shift by $1$ in $Q$.  This implies that $H$ is one-to-one and is close to identity in $C^2(Q)$.


Let $$\mu:=\frac{H_{\bar z}}{ H_z}$$ be the Beltrami differential in the rectangle $Q$ that corresponds to the preimage of the standard conformal structure under $H$.
Extend $\mu$ periodically to the strip $\tilde Q=\{|\Im z|\le  w\}$ so that  $\mu(z+1)=\mu(z)$. Extend $\mu$ by zero to the rest of $\bbC$.
If $f$ is close to $R_\alpha$, then $\|\mu\|_{L_\infty(\bbC)}$ is small.

  \smallskip

\noindent
{\sl Step 1.3: Beltrami equation in $\bbC$}
\\
Consider the solution $F$ to the Beltrami equation $F_{\bar z}/F_z = \mu$ with $F(0)=0, F(1)=1, F(\infty)=\infty$. By Measurable Riemann Mapping Theorem, this equation uniquely defines a quasiconformal homeomorphism $F\colon \bbC \to \bbC$.
Since $\mu$ is $1$-periodic, $F$ conjugates the unit translation to itself.

  \smallskip

\noindent
{\sl Step 1.4: $\Psi(z) =  F  H^{-1} (z/L)$}
\\
 The map $\Theta:=F  H ^{-1}$ is quasiconformal and preserves the standard conformal structure  due to the construction of $F$, thus  is conformal.
 As we have noticed above, $F$ commutes with the shift $z\mapsto z+1$ and $H^{-1}$ conjugates $g$ to this shift, thus the map  $\Theta = FH^{-1}$ conjugates $g$ to $z\mapsto z+1$. 

Set $$\Psi = \Theta(z/L)$$ (see Fig.\ref{fig-cd}). Then the map $\Psi$ is conformal and conjugates $f^n$ to to $z\mapsto z+1$. It is defined  in  the curvilinear quadrilateral $R$ bounded by the segment $I = [-iw L, i w L]$, its image $f^n(I)$, and two straight segments joining their endpoints. We can also extend $\Psi$ by the dynamics of $f^n$ and $f^{-n}$ to $f^n(R), f^{-n}(R)$.

Note that the above implies that $F$ is a $C^2$-smooth map because $\Theta$ is holomorphic and $H$ is $C^2$-smooth.

  \smallskip

\noindent
\textbf{Step 2: For real-symmetric $f$, $\Psi$ is real-symmetric (item \eqref{it-real})}

\noindent
Note that in this case, $L$ is real, $H$ preserves the real axis and commutes with the involution $\mathcal I\colon z\mapsto\bar z$. Thus the conformal structure associated with $\mu$ also commutes with $\mathcal I$. We conclude that $F$ commutes with $\mathcal I$.  Since $\Psi = FH^{-1}(z/L)$, this implies the statement.

  \smallskip
\noindent
\textbf{Step 3: if $f$ is close to the rotation, then $\Psi$ is close to linear (item \eqref{it-lin}).}

As we noted above, if $f $ is close to $R_\alpha$, then $H$ is close to the identity in $C^2(Q)$, thus $\|\mu\|_{L^{\infty}}$ is small. This implies that $F$ is close to the identity in a bounded domain $Q$, so $\Psi(z)=FH^{-1}(z/L)$ is close to $z\mapsto z/l$ in $R$.

\smallskip

\noindent
\textbf{Step 4. Analytic dependence on parameter (item \eqref{it-an-depend}).}

\noindent
Suppose that $f$ depends analytically on $\zeta$. Then $L=f^n(0)$ depends analytically on $\zeta$. Hence  $H, \mu$ depend analytically on $\zeta$.   Due to Ahlfors-Bers-Boyarski Theorem \cite[Theorem 11]{Ahl-B}, the solution of the Beltrami equation $F$ depends analytically on $\zeta$ as an element of the Banach space $B_{R,p}$; since it is $C^2$-smooth and we only use $F$ in the bounded rectangle $Q$, we get that $F$ depends analytically on $\zeta$ as a $C^2$ map.

However $H^{-1}$ does not depend analytically on $\zeta$, because the Jacobian $(H^{-1})'_\zeta = -  H'_\zeta (H'_z)^{-1}$ will not satisfy Cauchy-Riemann conditions\footnote{We thank Arnaud Ch{\'e}ritat for this comment.}. Nonetheless, $\Theta = FH^{-1}$ depends analytically on $\zeta$,  due to the following lemma.
\begin{lemma}
 If two $C^2$-smooth maps $g$, $h$ in $\bbC$  depend analytically on  a parameter $\zeta$ and $gh^{-1}$ is a holomorphic map, then it also depends analytically on $\zeta$.
\end{lemma}
The proof of the lemma is a straightforward computation: we get $$\frac{d}{d\zeta}gh^{-1} = \frac{d}{d\zeta} g - \frac{d}{dz} (g h^{-1}) \cdot  \frac{d}{d\zeta}h,$$
which is a complex-linear map.

We conclude that $\Psi$ depends analytically on $\zeta$.

\end{proof}

Renormalization transformations were used by Yoccoz (in the real-symmetric case) and by Risler.
 In particular, Yoccoz used them to reduce Theorem \ref{th-Yoc1} to Theorem \ref{th-Yoc2}. We summarize the main step of the reduction in the following proposition.
  \begin{proposition}{\bf (Yoccoz)}
    \label{prop-Yoc}
    Let $\{s_k\}, s_k\to 0$, be a decreasing sequence of positive numbers. For any positive $\eps, q$ there exists $m$ such that for any $\alpha\in\cB_{\{s_k\}}$, the value of  $\tilde \eps = \eps / \{q_m \alpha\}$ satisfies $$\tilde \eps > q (\Phi(\alpha_m)+1).$$
      For any $\eps', \kappa$ there exists $\nu$  such that for any $\alpha\in  \mathcal B_{\{s_k\}}$, the renormalization transformation $\mathcal R_{\tilde \eps, q_m}$, with $\tilde\eps$, $m$ as above, takes a $\nu$-neighborhood of $R_\alpha$  in $\mathcal D_{\eps}$ to a $\kappa$-neighborhood of the rotation $R_{\alpha_m}$ in $\mathcal D_{\tilde \eps-\eps'}$.

      The chart $\Psi$ can be made uniformly close to linear by diminishing $\nu$.
    \end{proposition}
The first statement is an easy observation on the Yoccos-Brjuno function. The rest of the statement follows from the fact that the chart $\Psi$ is close to linear for $f$ close to the rotation, see item \ref{it-lin} of Proposition \ref{prop-Psi-estim}.  Note that Proposition \ref{prop-Yoc} does not assume real symmetry.

We will outline the reduction of Theorem \ref{th-Yoc1} to Theorem \ref{th-Yoc2} from \cite{Yoccoz2002}.
Use Proposition \ref{prop-Yoc} to guarantee that the circle map $\mathcal R_{\tilde \eps, n} f$ is defined in a strip of width $\tilde \eps-\eps' \gg \Phi(\alpha)+C_0$. Now, Theorem \ref{th-Yoc2} implies that the map $\mathcal R_{\tilde \eps, n} f$ is conjugate to the rotation in the strip of width $\tilde \eps \cdot 2/3$. Since $\Psi$ is close to linear, the initial map $f$ is conjugate to the rotation in the strip of width $\eps/2$.  This implies Theorem \ref{th-Yoc1}. 

We will use a similar argument when deriving Theorem \ref{th-submanif} from the Main Theorem.

\subsection{The choice of $n$ and the exceptional set $K$}
\label{sec-domain}
  For a number $\alpha\in \bbR/\bbZ$,
  we define $n(\alpha):=q_m$ where $q_m $ is the denominator of the first continued fraction convergent of $\alpha$ that satisfies
  $$0<q_m\alpha-p_m<0.01.$$  It is easy to see that $$0<\{\alpha n(\alpha)\}<0.01.$$
  \begin{remark}
The function $n(\alpha)$ defined above is locally constant everywhere except at countably many rational points that only accumulate at the rational numbers $p/q$ with $q\le 100$.
  \end{remark}
  \begin{proof}
    The function is locally constant at irrational points. Consider $\alpha_0=p/q$ with $q>100$. 
    Then $n(\cdot)$, if not continuous, has a jump discontinuity at $\alpha_0$. Indeed, note that $p/q$ is a continued fraction convergent of $\alpha\approx p/q$. Now, for $\alpha < p/q$, $\alpha\approx p/q$, the number $p/q$ determines both the continued fraction convergents of $\alpha$ that precede $p/q$ and the location of the points $(R_{\alpha})^s(0), 0\le s\le q,$ inside or outside of $(0,0.01)$. Since at least one of these points belongs to $(0, 0.01)$ (this follows from $q>100$), the inequality $0<\alpha q_l-p_l<0.01$ holds for some $q_l\le q$, and the value of $n(\alpha)$ is determined by this information. The same applies to $\alpha>p/q, \alpha\approx p/q$; thus $n(\alpha)$ is constant on both semi-neighborhoods of $p/q$.

    This leaves only the numbers $p/q$ with $q\le 100$ as possible accumulation points of the discontinuities of $n(\cdot)$, and it is trivial to see that the discontinuities, indeed, accumulate at each of these points.

  \end{proof}

  The points of discontinuity of $n(\alpha)$ will  constitute the set $K$ in Theorem \ref{th-main}. On the set $$ \mathcal T=\{R_\alpha \mid \alpha\in (\bbR/\bbZ)\setminus K\}$$ the function $n(\cdot)$ is continuous and locally constant.


  Recall that the domain of $\mathcal R$ will be a neighborhood $\mathcal U$ of $ \mathcal T$. We assume that this neighborhood is a union of disjoint neighborhoods of connected components of $\mathcal T$. Then $n(\cdot)$ extends as a continuous locally constant function on $\mathcal U$.

   For $f\in \mathcal U$, we write $n=n(f)$ and use a fundamental domain of $f^n$ in the renormalization construction for $f$ below.

  \subsection{Renormalization operator}

Fix any $\eps\in \bbR^+$.
Here we define the operator $\mathcal R\colon \mathcal D_{\eps}\to \mathcal D_{1.5\eps}$ \footnote{Our main result deals with its restriction $\mathcal R\colon \mathcal D_{\eps}\to \mathcal D_{\eps}$.}.




\begin{definition}
\textbf{Renormalization} of $f\in \mathcal D_{\eps}$ is
$$\mathcal R f := \mathcal R_{2\eps, \,n(f)}\, f|_{\Pi_{1.5\eps}}.$$
\end{definition}

%

 \begin{remark}
 For a rotation, its renormalization is again a rotation.
 \end{remark}

 By Proposition \ref{prop-Psi-estim}, the restriction
 $$\cR:\cD_\eps\to\cD_\eps$$
 is an analytic operator; by an application of Koebe Distortion Principle and considerations of equicontinuity, it is a compact operator on a neighborhood of rotations. 
 As a first step towards proving hyperbolicity of $\cR$, let us formulate the following Technical Lemma, which is a quantitative refinement of Proposition \ref{prop-Psi-estim}:

\subsection{Technical Lemma: chart $\Psi$}
Recall that $l=\{n\alpha\}$; due to the choice of $n$ in Sec. \ref{sec-domain}, we have $l<0.01$.
Recall that $f^n(0)-0 = L$.
Finally, recall that we have chosen $w=2\eps$ in the definition of $\mathcal R$; hence the curvilinear rectangle  $R$ is bounded by the segment $I=[-2i\eps L, 2i\eps L]$, its image  $f^n(I)$, and two straight segments joining their endpoints.   Put $$\tilde R:= R\cap \{|\Im z|<1.5\eps l\}.$$

\begin{techlemma}
\label{lem-Psi-estim}
Let $\eps$ be fixed.

There exists a neighborhood  $\mathcal U \subset \mathcal D_\eps$ of $\mathcal T$ such that
for each $f\in \mathcal U\subset \mathcal D_\eps$, the biholomorphic map $\Psi$ from Proposition \ref{prop-Psi-estim} satisfies the following.
\begin{enumerate}



\item \label{it-lin-ref}
For $\eps>1$ we have
$$\|\Psi(zL) - z\|_{C(R)} \le c\eps^{1-2/p} \dist_{C(\Pi_{0.1\eps})}(f^n, R_{n\alpha})$$
where $p>2$ and  $c$ are universal constants.



 \item \label{it-deriv}
 Consider a family $\zeta\mapsto f_\zeta\in \mathcal D_\eps$, $\zeta\in \bbC^m$, that depends analytically on $\zeta$ with $f_0=R_\alpha$.
 Then
  $$\|\Psi'_{\zeta}|_{\zeta=0}\|_{C(R)}<C(\eps) \|(f^n)'_{\zeta}|_{\zeta=0}\|_{C(\Pi_{\eps/2})},$$ and for $\eps>1$,
\begin{equation*}
 \label{eq-zzzeta}
 \|\Psi'''_{zz\zeta}|_{\zeta=0}\|_{C(\tilde R)}< c \eps^{-0.5}l^{-2} \|(f^n)'_{\zeta}|_{\zeta=0}\|_{C(\Pi_{\eps/2})},
\end{equation*}
 where $c$ is a universal constant.

\item \label{it-conj} Suppose that $f$ is analytically conjugate to a rotation,  $f = \xi R_{\alpha} \xi^{-1}$ for some biholomorphic map $\xi\colon \Pi_{\eps/2}\to \bbC/\bbZ$. We do not assume $f\in \mathcal U$.

There exists $\theta=\theta(\eps)$ such that if $f$ is conjugate to a rotation as above and
$$ \|\xi - id\|_{C^1(\Pi_{\eps/2})} < \theta,$$
then $\Psi$ is well-defined for all annuli  maps close to $f$ and satisfies assertions of Proposition \ref{prop-Psi-estim}. Moreover,
$$
\|\Psi(Lz) - z\|_{C(R)} \le c(\eps) \theta.
$$

\end{enumerate}

\end{techlemma}
  Item \eqref{it-lin-ref} is a quantitative version of item \eqref{it-lin} in Proposition \ref{prop-Psi-estim}. Roughly speaking, item \eqref{it-deriv} strengthens this statement.  Namely, \eqref{it-lin-ref} shows that $\Psi$ is close to linear; \eqref{it-deriv} is an infinitesimal version of the same statement, and can be understood as  a bound on the derivative of the mapping $f\mapsto \Psi$ both in the standard metric and in the metric $\sup \|g''\|$ on the tangent space to the image. The latter metric  corresponds to the nonlinearity of $\Psi$. The term $l^{-2}$ in the last inequality in \eqref{it-deriv} is typical for second derivatives of maps that are close to linear expansions $z\mapsto z/l$.

\begin{proof}

Below we use the same notation as in Proposition \ref{prop-Psi-estim}.

$\;$ \\\noindent
\textbf{If $f$ is sufficiently close to $R_\alpha$ for a fixed $\alpha$, then $\Psi$ is close to $z\mapsto z/l$ (item \eqref{it-lin-ref}).}


Suppose that $\dist_{C(\Pi_{0.1\eps})} (f^n, R_{n\alpha}) =\kappa$.
We will show that $\mu$ is $c\kappa$-close to zero and use an appropriate estimate from  \cite{Ahl-B}. In this step, $c$ is an arbitrary universal constant.

It is easy to prove that $g (z) = f^n(Lz)/L$ is $c\kappa$-close to the shift by $1$ in $C^2(I)$ for $\eps>1$. Indeed,
$$\dist_{C(R/L)}(g', 1) = \dist_{C(R)} ((f^n)', 1) \le c\kappa\eps^{-1},$$
$$\dist_{C(R/L)}(g'', 0) = L \dist_{C(R)} ((f^n)'', 0) \le c\kappa\eps^{-2}.$$
due to Cauchy estimates for $f^n$.
The estimate on $g'$ implies that for any $\eps$, $g$ is $c\kappa$-close to the shift by $1$ on $I = [-2i\eps, 2i\eps]$ because $g(0)=1$. Finally, $g$ is $c\kappa$-close to $z\mapsto z+1$ in $C^2(I)$, thus $$\dist_{C^2(Q)} (H, \text{id}) \le c\kappa, \qquad \|\mu\|_{L_\infty} \le c\kappa.$$

Now, Theorem 8 in \cite{Ahl-B} implies that for a $\mu$-conformal homeomorphism $F$ on the plane with $F(0)=0,F(1)=1, F(\infty)=\infty$, we have that
$$\|F-\text{id}\|_{B_{R,p}} \le c\|\mu\|_{\infty}$$
where $c$, $p>2$ are  universal constants and  $B_{R,p}$ is a Banach space of functions in $D_R=\{|z|<R\}$ with the metric
$$\|w\|_{B_{R, p}} = \sup \frac{|w(z_1)-w(z_2)|}{|z_1-z_2|^{1-2/p}} + \left(\int_{D_R} |w|^p dxdy\right)^{1/p}.$$
Choosing $z_1=z, z_2=0$, we get $$\|F(z)-z\| \le c |z|^{1-2/p} \|\mu\|_\infty \le c \eps^{1-2/p} \kappa$$
in $Q = [0,1]\times[-2i\eps, 2i\eps]$ provided that $\eps>1$. Since $\Psi(z) = FH^{-1}(z/L)$, the result follows: $$|\Psi(Lz) - z|<c \eps^{1-2/p} \kappa\text{ in }R\text{ for }\eps>1.$$


\smallskip

\noindent

$\;$ \\\noindent
\textbf{Estimates on the derivative of $\Psi$ with respect to the parameter (item \eqref{it-deriv})}

\noindent
As before, let $c$ denote any universal constant. Let $C(\eps)$ be any constant that depends on $\eps$ only. All functions $f, g, H, \mu, F, \Theta, \Psi$ and a constant $L=f^n(0)-0$ now depend on $\zeta$.

\begin{lemma}
\label{lem-g}
 The map $g (z)= f^n(Lz)/L $ satisfies
 $$\|g'_{\zeta}|_{\zeta=0}\|_{C(I)}\le c \|(f^n)'_{\zeta}|_{\zeta=0}\|_{C(\Pi_{\eps/2})}$$
 and
 $$\|g'_{\zeta}|_{\zeta=0}\|_{C(\Pi_\eps)}\le c \eps^{-1} \max(1,\eps)\|(f^n)'_{\zeta}|_{\zeta=0}\|_{C(\Pi_{\eps/2})},$$ $$  \|g''_{z\zeta}|_{\zeta=0}\|_{C(\Pi_\eps)}\le c \eps^{-1} \|(f^n)'_{\zeta}|_{\zeta=0}\|_{C(\Pi_{\eps/2})}.$$
\end{lemma}
\begin{proof}
Since
$ g'(z) = (f^n)'(zL)$, we have
$$ g''_{z\zeta} = (f^n)''(zL) \cdot zL_\zeta +  (f^n)''_{z\zeta}(zL); \qquad  g''_{z\zeta}|_{\zeta=0} = (f^n)''_{z\zeta}|_{\zeta=0}(lz).$$
Due to Cauchy estimates, $\|(f^n)''_{z\zeta}|_{\zeta=0}\|_{C(\Pi_{l\eps})} \le c \eps^{-1} \|(f^n)'_{\zeta}|_{\zeta=0}\|_{C(\Pi_{\eps/2})}$, so
$$\|g''_{z\zeta}|_{\zeta=0} \|_{C(\Pi_\eps)}\le c \eps^{-1} \|(f^n)'_{\zeta}|_{\zeta=0}\|_{C(\Pi_{\eps/2})},$$ and since $g'_{\zeta}(0)=0$, this implies
$$|g'_{\zeta}|_{\zeta=0} |\le c \|(f^n)'_{\zeta}|_{\zeta=0}\|_{\Pi_{\eps/2}}\text{ on }I.$$
$$|g'_{\zeta}|_{\zeta=0} |\le c\max (1, \eps)\eps^{-1} \|(f^n)'_{\zeta}|_{\zeta=0}\|_{C(\Pi_{\eps/2})}\text{ on }\Pi_\eps.$$

\end{proof}

\begin{lemma}
In the above notation,  $$\|\mu'_{\zeta}|_{\zeta=0}\|_{\infty} \le C(\eps) \|(f^n)'_{\zeta}|_{\zeta=0}\|_{C(\Pi_{\eps/2})}. $$ Moreover, $C(\eps)$ is bounded as $\eps \to \infty$.
\end{lemma}
\begin{proof}

Note that
$$
H'_\zeta = t g'_{\zeta}(is), \qquad 2 H''_{\zeta \bar z} = g'_{\zeta} (is) + t g''_{z\zeta}(is).
$$
 Due to Lemma \ref{lem-g},
\begin{equation}
  \label{eq-H-zeta}
 \max\left(     \|H'_{\zeta}|_{\zeta=0}\|_{C(Q)},\|H''_{\zeta \bar z}|_{\zeta=0}\|_{C(Q)}\right) \le  C(\eps) \|(f^n)'_{\zeta}|_{\zeta=0}\|_{C(\Pi_{\eps/2})}
\end{equation}
where $C(\eps)$ is bounded as $\eps\to \infty$.

\noindent
So
$$\mu'_\zeta = \left(\frac{H'_{\bar z}}{H'_z} \right)'_{\zeta} = \frac{H''_{\zeta \bar z}}{H'_z} - \frac{H'_{\bar z} H''_{z\zeta}}{(H'_z)^2}.$$
The second summand is zero for $\zeta=0$: $H_0(z)=z$ is holomorphic. The first summand is estimated from above by $C(\eps) \|(f^n)'_{\zeta}|_{\zeta=0}\|_{C(\Pi_{\eps/2})} $.

\end{proof}
Note that $\mu$ is infinitely differentiable with respect to $\zeta$ around $\zeta=0$,  and  each derivative is bounded on $\bbC$ (because $\mu$ is periodic in $\tilde Q$ and zero elsewhere). So the following  lemma from \cite{Ahl-B} applies to $\mu$:
\begin{lemma}[Lemma 22, \cite{Ahl-B}]
\label{lem-deriv-Beltrami}
For all $t$ in some open set $\Delta$, suppose that
$$
\mu(t+s) = \mu(t) + a(t) s + |s|\alpha(t,s)
$$
with $\|\alpha(t,s)\|_\infty \le c$ and $\alpha(t,s) \to 0$ almost everywhere as $s\to 0$. Suppose further that the norms $\|a(t+s)\|$ are bounded and that $a(t+s)\to a(t)$  almost everywhere for $s\to 0$. Then the solution $F^{\mu (t)}$ of the equation $F_{\bar z}/F_z = \mu(t)$ with  $F(0)=0, F(1)=1, F(\infty)=\infty$ has a development
$$F^{\mu(t+s)} = F^{\mu(t)} + \theta(t) s + |s| \gamma(t,s)$$
where $\|\gamma(t,s)\|_{B(R,p)}\to 0$ for $s\to 0$.
The derivative $\theta$ satisfies
$$|\theta(z)|\le c (1+|z|)\log(1+|z|) \|a\|_\infty.$$
\end{lemma}

Due to our estimate on $\mu'_\zeta$, this lemma implies that $F$ is differentiable with respect to $\zeta\in \bbC$ at zero, and  $F'_\zeta(z) \le C(\eps) \|(f^n)'_{\zeta}\|_{C(\Pi_{\eps/2})}$. Since
\begin{equation}
\label{eq-Theta}
 \Theta'_{\zeta}|_{\zeta=0} = (FH^{-1})'_{\zeta}|_{\zeta=0} = F'_{\zeta}|_{\zeta=0} (z) +  (H^{-1})'_{\zeta}|_{\zeta=0},
\end{equation}
  $$ \Psi'_{\zeta}|_{\zeta=0} = \Theta'_{\zeta}|_{\zeta=0} (z/L) - z L'_{\zeta}|_{\zeta=0}/L^2, \qquad L'_{\zeta}|_{\zeta=0} = (f^n)'_{\zeta}|_{\zeta=0}(0),$$
 we use \eqref{eq-H-zeta} to get $$\|\Psi'_{\zeta}|_{\zeta=0}\|_{C(R)} \le C(\eps, \alpha) \|(f^n)'_{\zeta}|_{\zeta=0}\|_{C(\Pi_{\eps/2})}.$$

 Now estimate $\Psi'''_{zz\zeta}|_{\zeta=0}$ for $\eps>1$. Recall that  $|F'_{\zeta}(z)|$  is estimated in Lemma \ref{lem-deriv-Beltrami} above by $c  (1+\eps)\log(1+\eps)  \|(f^n)'_{\zeta}\|_{C(\Pi_{\eps/2})}$ in the disc $|z|<\eps$ that contains $Q$. So
\begin{equation}
\label{eq-F-zeta}
\|F'_{\zeta}\|_{C(Q)} \le c\eps^{1.5} \|(f^n)'_{\zeta}\|_{C(\Pi_{\eps/2})}.
\end{equation}
Using \eqref{eq-Theta} and \eqref{eq-H-zeta} again, we get
 $$\Theta'_{\zeta}|_{\zeta=0} \le c\eps^{1.5}  \|(f^n)'_{ \zeta}\|_{C(\Pi_{\eps/2})}$$ for  $\eps>1$.

For small $\zeta$, extend $\Theta$ via iterates of $g$ to the rectangle $\hat Q = [-\eps,\eps]\times[-2i\eps, 2i\eps]$; we need approximately $\eps$ iterates. Each iterate increases the estimate on $\Theta'_\zeta$ by $g'_{\zeta}$. Due to  Lemma \ref{lem-g}, on the rectangle $\hat Q$, we have $|g'_{\zeta} |\le c\|(f^n)'_{ \zeta}\|$. Thus the estimate on $\Theta'_\zeta$ in $\hat Q$ is $$c\eps^{1.5} \|(f^n)'_{ \zeta}\| + c\eps \|(f^n)'_{ \zeta}\|.$$

Now we use Cauchy estimates for $\Theta'_{\zeta}$ in the rectangle  $$\tilde Q = [0,1]\times [-1.5i\eps, 1.5i\eps] \subset \hat Q = [-\eps,\eps]\times[-2i\eps, 2i\eps]$$ to get $$\Theta'''_{\zeta zz}|_{\tilde Q} \le \|(f^n)'_{ \zeta}\|_{C(\Pi_{\eps/2})} \frac {c\eps^{1.5} + c\eps}{\eps^2} < c\eps^{0.5} \|(f^n)'_{ \zeta}\|_{C(\Pi_{\eps/2})}$$  for  $\eps>1$.

Since $\Psi = \Theta (z/L)$, we have $\Psi''_{zz} = L^{-2} \Theta''(z/L)$, thus
\begin{eqnarray*}
  \Psi'''_{zz\zeta}|_{\zeta=0} (z)&=& \Theta'''_{zz\zeta}|_{\zeta=0}(z/l) l^{-2} + \Theta'''(z/L) (z/L)'_\zeta L^{-2} - \Theta''(z/L) 2L'_\zeta/L^3 \\
  &=& \Theta'''_{zz\zeta}|_{\zeta=0}(z/l) l^{-2},\end{eqnarray*}
because $\Theta=\text{id}$ for $\zeta=0$.
So the above estimate implies item \eqref{it-deriv} of the Technical lemma.

\end{proof}

\noindent
\textbf{The case when $f$ is analytically conjugate to a rotation, item \eqref{it-conj}}

Suppose that  $f$ is conjugate to $R_\alpha$ via $\xi$, $|\xi-id|_{C_1(\Pi_{\eps/2})}<\theta$, and $\xi(0)=0$. We will show that $g(z) = f^{n}(Lz)/L$ is $c\theta$-close to the unit translation in $C^2(Q)$, where $c=c(\eps)$ does not depend on $\alpha$. Then the proof of Proposition \ref{prop-Psi-estim} applies to $f$ without any modifications.

Note that $g= f^n(Lz)/L$ is conjugate to $R_{\{n\alpha\}/ L}$ via $\xi(Lz)/L$. The latter conjugacy is $c(\eps) \theta$-close to identity in $C(\Pi_{\eps/(2l)})$. Indeed,
$$|L^{-1}\xi (Lz) - z| = |\frac{\xi(Lz) - \xi(0)}{L} -z | = |z \xi'(\eta) -z| <|z|\theta<\max (1, \eps) l\theta.$$
We conclude that $g$ is $c(\eps)\theta$-close to $R_{\{n\alpha\}/L}$ in $C(\Pi_{\eps/(2l)})$. Note that $\{n\alpha\}/L = \{n\alpha\}/(f^n(0)-0)$, and
$$
|f^n(0)-0| = |\xi R_{\{n\alpha\}}\xi^{-1}(0) - \xi(\xi^{-1} (0))| = |\xi'(\eta)|\, \{n\alpha\}.
$$
Thus the size of the translation is $c(\eps) \theta$-close to $1$. Finally, $g$ is $c(\eps)\theta$-close to the unit translation in $C(\Pi_{\eps/(2l)})$, thus is $c(\eps)\theta$-close to the unit translation in $C^2(Q)$.

Since in Proposition \ref{prop-Psi-estim} we only use  that $g$ is close to the unit translation in $C^2(Q)$, the same arguments as above show that for sufficiently small $\theta$, the map  $\Psi$ is well-defined for diffeomorphisms sufficiently close to $f = \xi R_\alpha \xi^{-1}$  and satisfies assertions of Proposition \ref{prop-Psi-estim}. The estimate $\|\Psi(Lz) - z\|_{C(R)} \le c(\eps) \theta$ is established in the same way as in item \eqref{it-lin-ref}.

%% file: operator.tex
\section{Renormalization as an analytic operator of Banach affine manifolds}

\subsection{Domain of definition of $\mathcal R$}

Recall that $P$ is the first-return map to the fundamental domain $R$ under the action of $f$.
\begin{lemma}
  \label{lem-domain1}
 For fixed $\eps$, for each $\alpha\notin K$, there exists a neighborhood $U_\alpha$ of $R_\alpha$ in $\mathcal D_\eps$ such that  for $f\in U_\alpha$, the map $\Psi P \Psi^{-1} $  is a well-defined holomorphic map in $\Pi_{1.5\eps}$.

There exists $\theta=\theta(\eps)$ such that if $f$ is analytically conjugate to the rotation,  $f = \xi  R_{\alpha} \xi^{-1}$, with
$$\xi\colon \Pi_{\eps/2}\to \Pi_{\eps},\qquad  \|\xi - id\|_{C^1(\Pi_{\eps/2})} < \theta,$$
then in some neighborhood of $f$ in  $\mathcal D_\eps$ the map $\Psi P \Psi^{-1} $  is a well-defined holomorphic map in $\Pi_{1.5\eps}$.
Note that we do not assume that $f\in U_\alpha$ here.

If $f$ preserves the real axis, then $\Psi P \Psi^{-1} $ also preserves the real axis.
\end{lemma}
\begin{proof}
If $f$ is sufficiently close to $R_\alpha$, then the first-return map $P$ to $R$ is well-defined and univalent in $\Pi_{1.9l\eps}\cap R$.
Proposition \ref{prop-Psi-estim}  shows that if $f$ is close to $R_{\alpha}$, then   $\Psi$ is close to the linear map  $z\mapsto z/l$ in $R$. So $\Psi P \Psi^{-1}$ is defined in the rectangle $\Psi (\Pi_{1.9 l\eps}\cap R)$ of width $1$ and height approximately $1.9 \eps$. It extends to a  well-defined holomorphic map in the strip $\Pi_{1.5\eps}$ because the map $\Psi$ conjugates $f^n$ to the unit translation.

If $f=\xi R_\alpha \xi^{-1}$ and $\xi$ is $\theta$-close to identity for sufficiently small $\theta$, then the first-return map $P$ is automatically well-defined in $R$ except $0.1 \eps l$-neighborhoods of its borders. Same holds for any map that is sufficiently close to $f$.  Further, for any map close to $f$, $\Psi$ is well-defined in $R$ due to item \eqref{it-conj} of the Technical lemma, and $c\theta/L$-close to the linear map $z\mapsto z/l$. So for all maps in a neighborhood of $f$ in $\mathcal D_\eps$,  $\Psi P \Psi^{-1} $  defines a holomorphic map in a rectangle $\Psi (\Pi_{1.9 l\eps}\cap R)$ of width $1$ and height approximately $1.9 \eps$. As above, it extends to  a  well-defined holomorphic map in the strip $\Pi_{1.5\eps}$.

If $f$ preserves the real axis, then $\Psi$ preserves the real axis due to item \eqref{it-real} of Proposition \ref{prop-Psi-estim}, thus  $\Psi P \Psi^{-1} $ preserves the real axis.

\end{proof}
With $U_\alpha$ as above, let us define
$$\mathcal U := \cup U_\alpha,\;\alpha\in\mathcal T\subset \mathcal D_\eps.$$

Recall that we define the renormalization operator in the following way:
$$\mathcal R f := \widetilde \Psi P \widetilde \Psi^{-1}|_{\Pi_{1.5\eps}},$$
where the map  $\widetilde \Psi\colon R/f^n \mapsto \bbC/\bbZ$ is induced by the map $\Psi\colon R\to \bbC$ constructed in Proposition \ref{prop-Psi-estim}, and $P$ is the first-return map to the  domain $R$ under the action of $f$.

\begin{lemma}
\label{lem-domain2}
  In the previous notation, the following holds:
\begin{enumerate}
\item $\mathcal R$ is a real-symmetric complex-analytic operator $\mathcal R \colon \cU\to \mathcal D_{1.5\eps};$

\item There exists $\theta=\theta(\eps)$ such that if $f\in \mathcal D_\eps$ is analytically conjugate to the rotation,  $f = \xi R_{\alpha} \xi^{-1}$, with
$$\xi\colon \Pi_{\eps/2}\to \Pi_{\eps},\qquad  \|\xi - id\|_{C^1(\Pi_{\eps/2})} < \theta,$$
then for some neighborhood $\mathcal U_f$ of $f$, $\mathcal R\colon \mathcal U_f \to \mathcal D_{1.5\eps}$ is  a real-symmetric complex-analytic operator.

  \item For $\alpha\in\cT$, the differential $$(\mathcal R)'|_{R_\alpha} \colon T_{R_{\alpha}}\mathcal D_{\eps} \mapsto T_{\mathcal R(R_{\alpha})}\mathcal D_{\eps}$$
  of the restriction $\mathcal R \colon \mathcal U \to \mathcal D_\eps$ to $\mathcal D_\eps$  is a compact linear operator.
\end{enumerate}
\end{lemma}

\begin{proof}
  The first two statements follow from Lemma~\ref{lem-domain1} and item \eqref{it-an-depend} of Proposition \ref{prop-Psi-estim}. The third statement follows since the differential $(\mathcal R)'|_{R_\alpha} \colon T_{R_{\alpha}}\mathcal D_{\eps} \mapsto T_{\mathcal R (R_{\alpha})}\mathcal D_{1.5\eps}$ is a bounded linear operator, and the restriction of a vector field to a smaller strip $T_{\mathcal R(R_{\alpha})}\mathcal D_{1.5\eps} \mapsto T_{\mathcal R(R_{\alpha})}\mathcal D_{\eps} $ is a compact operator.
\end{proof}

\section{Infinitesimal expansion/contraction of the differential $\mathcal R'$}

\subsection{Unstable direction of $\mathcal R'$}

The tangent space to the Banach affine manifold $\mathcal D_{\eps}$ at $R_{\alpha}$ is the space of analytic vector fields in $\Pi_{\eps}$.
It is the sum of a (complex) one-dimensional subspace $V_1$ generated by a unit vector field $d/dz$
and a (complex) codimension-1 closed linear subspace $V_0$ of vector fields with zero average, $\int_{\bbR/\bbZ}v \, dz=0$. Note that $\mathcal R$ takes a rigid rotation to a rigid rotation (but with a different rotation number), thus the real one-dimensional space formed by rotations is invariant under $\mathcal R$.

This implies that $V_1$ is invariant under $\mathcal R'|_{R_\alpha}$, $\alpha\in \bbR$.
\begin{lemma}
\label{lem-expand}
For all real $\alpha$,  $\mathcal R'|_{R_\alpha}$ expands on $V_1$.
\end{lemma}
\begin{proof}
Let  $\mathcal R R_\alpha=R_\beta$. Let us compute $\beta$.
For $f=R_\alpha$, $R$ intersects the real axis on $[0, f^n(0)] = [0,\{n\alpha\}]$ where $n=q_m$ for some $m$. The next return of the orbit of zero under $R_\alpha$ to this interval is $\{(q_m+q_{m+1})\alpha\} $. Thus $P$ takes zero to $(q_m+q_{m+1})\alpha - (p_m+p_{m+1}) $, and $\mathcal R R_\alpha$ takes zero to
 $$\beta=\frac{(q_m+q_{m+1})\alpha - (p_m+p_{m+1})}{q_m\alpha-p_m}  = \frac{q_{m+1}\alpha-p_{m+1}}{q_m\alpha-p_m} \mod 1.$$
 The derivative of the latter expression with respect to $\alpha$ is  $$\frac{q_{m+1}(q_m\alpha-p_m)}{(q_m\alpha-p_m)^2} - q_m\frac{(q_{m+1}\alpha-p_{m+1})}{(q_m\alpha-p_m)^2} = \frac{-q_{m+1}p_m +q_mp_{m+1}}{(q_m\alpha-p_m)^2}  = \frac{\pm 1}{ (q_m\alpha-p_m)^2}.$$
 Since the denominator is $l^2 \le 10^{-4}$, the result follows.
\end{proof}

\subsection{Stable subspace for $\mathcal R'$}
The key lemma is the following.
\begin{lemma}
\label{lem-curve}
For any $\alpha$ that satisfies the Brjuno condition, any vector field $v\in V_0\subset T\mathcal D_\eps$ is tangent to a  curve $\zeta\mapsto f_\zeta, f_0=R_\alpha$, $\zeta\in \bbC$, such that all $f_\zeta$ are analytically conjugate to $R_\alpha$ in $\Pi_{\eps/2}$.
\end{lemma}

After we prove this lemma, we are going to show that $\mathcal R$ contracts on  such families $f_\zeta$.

Note that the curve
$$
f_\zeta = (\text{id} + \zeta h) R_\alpha (\text{id}+\zeta h)^{-1}
$$
is tangent to $v$ at $\zeta=0$ if and only if the vector field $h$ satisfies
\begin{equation}
    \label{hom-eq}
    h(z+\alpha)-h(z)=v(z).
\end{equation}
So Lemma \ref{lem-curve} is implied by the following lemma.

\begin{lemma}
\label{lem-tangent}
Suppose that $\alpha, \eps$ satisfy
\begin{equation}
  \label{condition-qn}
  \frac{\log q_{k+1}}{q_{k}}  <2\pi 0.05\eps
\end{equation}
for sufficiently large $k$.
For each analytic vector field $v$ in $\Pi_{\eps}$ with zero average, there exists a unique analytic  vector field $h$ in $\Pi_{0.9\eps}$ such that $h(0)=0$ and
(\ref{hom-eq}) holds for $h$, $v$.
Moreover,  $$\|h\|_{C(\Pi_{0.9\eps})}<c(\alpha,\eps)\|v\|_{C(\Pi_\eps)}\text{ for some constant }c(\alpha, \eps).$$

\end{lemma}
In particular, for any $\eps$, the inequality \eqref{condition-qn} will hold for any  $\alpha\in \mathcal B$  for sufficiently large $k$, since the Brjuno function $\Phi_0(\alpha)$ is finite for $\alpha\in \mathcal B$.

\begin{proof}
We will achieve $h(0)=0$ by shifting the argument.

Note that if a function $f$ is analytic in the strip of width $\eps$, then its Fourier coefficients decrease exponentially fast:
$$|a_n| \le  \frac{\|f\|_{C(\Pi_\eps)}}{e^{2\pi \eps |n|}}.$$ The converse is also true: if $|a_k|<ce^{-2\pi k\eps}$ for some $c$, then the function is analytic in any strip of width smaller than $\eps$.

Suppose that $v$ has Fourier coefficients $a_k$; note that $a_0=0$ because $v$ has zero average. Then  $h$ satisfies (\ref{hom-eq}) if and only if it has Fourier coefficients $b_k  = a_k / (e^{2\pi i k\alpha}-1)$ for $k\neq 0$. Denominators are non-zero because $\alpha$ is irrational and $k\neq 0$.

To prove analyticity of $h$ in the strip of width $0.9\eps$, we will prove that
\begin{equation}
\label{eq-estim}
|e^{2\pi i k \alpha}-1|>e^{-2\pi |k| 0.09\eps}
\end{equation}
for sufficiently large  $|k|$. Clearly, this implies the statement.

Note that $$\frac{1}{\pi}<\frac{|e^{2\pi i k \alpha}-1|}{\dist (k\alpha, 1)}<1.$$  So it is sufficient to estimate  $\dist( k \alpha, 1)$. We will only consider $k>0$; the case $k<0$ is symmetric.

Recall that the approximants $p_n/q_n$ correspond to closest returns to zero for $n\alpha\mod 1$:
$$\dist (k\alpha, 1)>\dist (q_n\alpha, 1)\text{ for }q_n +1 \le k< q_{n+1}.$$
Thus it is sufficient to prove that
$$\dist( q_n \alpha, 1)>e^{-2\pi q_{n} 0.08\eps}.$$

The left-hand side is between $\frac{1}{q_{n+1}}$ and $\frac{1}{2q_{n+1}}$. Thus it suffices to prove $2q_{n+1} < e^{2\pi q_n 0.08\eps}$ for sufficiently large $n$.
But this is implied by the inequality $$\frac{\ln q_{n+1}}{q_n}<2\pi 0.05\eps$$ for sufficiently large $n$.
We have proved \eqref{eq-estim}, thus analyticity of $h$ in $\Pi_{0.9\eps}$.

Now let us estimate $\|h\|_{\Pi_{0.9\eps}}$. Recall that $$|a_n| \le  \|v\|_{C(\Pi_\eps)}e^{-2\pi \eps |n|}$$ and $h(z) = \sum_k \frac{a_k}{e^{2\pi i k\alpha} -1} e^{2\pi i kz}$. Thus for $|\Im z| < 0.9 \eps$,
$$
 \|h\|_{\Pi_{0.9\eps}}<\|v\|_{\Pi_\eps} \sum_k  \frac{e^{-2\pi 0.1\eps |k|}}{|e^{2\pi i k\alpha} -1|}   =
c(\alpha, \eps) \|v\|_{\Pi_\eps}.
$$
Here $c(\alpha, \eps)$ is finite due to \eqref{eq-estim}.

\end{proof}

\subsection{Contraction on $V_0$ for large $\eps$}

Note that $\sup_{\Pi_\eps} |v''|$ is a norm on $V_0$; we denote it by $\|v\|_2$. Note also that this norm is equivalent to $\|v\|$; indeed, one can easily prove
$$\|v\| \le (\max (1, \eps))^2\|v''\|\text{ and }\|v''\| \le \frac {1}{2\pi (\min (1, \eps))^2}\|v\|.$$

 For a real $\alpha$, consider the linear operator $$\mathcal L_\alpha \colon h \to h(z+\alpha)-h(z)$$ on the space $T\mathcal D_{0.9\eps}$. Define the linear operator $\mathcal M_\alpha \colon T\mathcal D_{\eps} \to T\mathcal D_{0.9\eps}$ that takes $v$ to $h$ provided by Lemma \ref{lem-tangent}. This operator is ``inverse'' to $\mathcal L_\alpha$: $\mathcal L_\alpha \mathcal M_\alpha v$ is a restriction of $v$ to $\Pi_{0.9\eps}$.
Clearly, $\|\mathcal L_\alpha\|\le 2$; in the  assumptions of the previous lemma, $$\|\mathcal M_\alpha\|_{C(\Pi_{0.9\eps})}<c(\alpha, \eps)\|v\|_{C(\Pi_{\eps})}.$$

Roughly speaking, we will prove that the operator $\mathcal M_\beta \mathcal R \mathcal L_\alpha$ contracts, where $\beta$ is such that $\mathcal R R_\alpha = R_\beta$.   However we will work with derivatives rather than operators themselves.


The following lemma gives an explicit form of $(\mathcal M_\beta \mathcal R \mathcal L_\alpha)'$.
\begin{lemma}
Let $\alpha$ satisfy assumptions of Lemma \ref{lem-tangent}, let $v\in V_0$, $h := \mathcal M_\alpha v$, $\xi_\zeta := id + \zeta h$, and $f_\zeta := \xi_\zeta R_\alpha \xi_\zeta^{-1}$. Put $R_\beta=\mathcal R (R_\alpha)$. Then
 $$\mathcal R' v = \mathcal L_\beta \mathcal Q \mathcal M_{\alpha} v,$$ where
$$
\mathcal Q h:= \Psi'_{\zeta}|_{\zeta=0} (lz) + h (lz)/l,
$$
and $\Psi=\Psi_\zeta$ corresponds to  $f_\zeta$.
\end{lemma}
\begin{proof}
Since $h = \mathcal M_\alpha v$, $f_\zeta$ is tangent to $v$, so  $$\mathcal R'v = (\mathcal R f_\zeta)'_{\zeta}|_{\zeta=0}.$$

Since $f_\zeta$ is analytically conjugate to $R_\alpha$ via $\xi$, $\mathcal R f_\zeta$ is conjugate to the rotation $R_\beta=\mathcal R R_\alpha$ via $\hat \xi = \Psi \xi (lz)$. Note that
$$(\hat \xi_\zeta)'_{\zeta}|_{\zeta=0} = \Psi'_{\zeta}|_{\zeta=0} (lz) + h (lz)/l = \mathcal Q h.$$
Finally, since $\mathcal R f_\zeta = \hat \xi R_\beta \hat \xi^{-1} $, we have
$$\mathcal R'v = (\mathcal Rf_\zeta)'_{\zeta}|_{\zeta=0} = \mathcal Qh(z+\beta)-\mathcal Qh(z) = \mathcal L_{\beta} \mathcal Q h  =  \mathcal L_\beta \mathcal Q \mathcal M_\alpha v.$$
\end{proof}

\begin{lemma}
\label{lem-Q}
If $\eps$ is larger than a universal constant, then
 $$\|(\mathcal Qh)''\|_{C(\Pi_{0.9 \eps})}\le  0.1 \|h''\|_{C(\Pi_{0.9\eps})}.$$
\end{lemma}
\begin{proof}
We have
$$
(\mathcal Qh)'' = l^2\Psi'''_{zz\zeta}|_{\zeta=0}(lz) + l h''(lz).
$$
The second summand is estimated by $l\|h''\|_{C(\Pi_{0.9\eps})}$. Further, for $\eps>1$,
$$\|\Psi'''_{zz\zeta}|_{\zeta=0} \|_{C(\tilde R)}\le c \eps^{-0.5}l^{-2}\|(f^n)'_{\zeta}|_{\zeta=0}\|_{C(\Pi_{\eps/2})}$$
due to the Technical lemma, see item \eqref{it-deriv}; note that $z\in \Pi_{0.9\eps}, \Re z \in [0,1]$ implies $lz \in \tilde R$. Now, at $\zeta=0$,
$$(f^n)'_{\zeta} = \xi'_\zeta(z+n\alpha)-\xi'_{\zeta}(z) =  h(z+n\alpha)-h(z) = \{n\alpha\} h'(\nu).$$
Since $h$ is 1-periodic,  integrals of $h'$ over circles $\bbR/\bbZ + i\theta$ equal zero; so $\|h'\|_{C(\Pi_{\eps/2})} \le \|h''\|_{C(\Pi_{\eps/2})}$. Finally,  $$\|\Psi'''_{zz\zeta}|_{\zeta=0}\|_{C(\tilde R)} \le c \eps^{-0.5}l^{-1}\|h''\|_{C(\Pi_{\eps/2})}.$$
So for sufficiently large $\eps$, the statement holds.
\end{proof}

\begin{theorem}
  \label{th:infcont}
 For  any $\alpha\in \mathcal B$,  for $\eps$ is larger than a universal constant, the differential of the operator $\mathcal R\colon \mathcal D_\eps\to \mathcal D_\eps$ at $R_\alpha$ contracts on the space of vector fields with zero average: for all $N$, for $v\in V_0$,
 $$
\|(\mathcal R^N)'v \|< C (0.1)^N  \|v\|,
$$
where $C$  depends on $\eps, \alpha$ only.

\end{theorem}
\begin{proof}
We assume that $\eps$ is sufficiently large so that it satisfies assumptions of Lemma \ref{lem-Q}. Since $\alpha\in \mathcal B$, assumptions of Lemma \ref{lem-tangent} are satisfied.

Since $\|\mathcal L_\alpha v\|_2\le 2 \|v\|_2$,   $\|\mathcal M_\alpha v\|_{C(\Pi_{0.9\eps})} <c(\alpha, \eps)\|v\|_{C(\Pi_\eps)}$, and norms $\|\cdot \|$ and $\|\cdot\|_2$ are equivalent, Lemma \ref{lem-Q} implies that
$$(\mathcal R^N)' = \mathcal L_{\beta_n} \mathcal Q_{N} \dots \mathcal Q_1 \mathcal M_{\alpha}$$
satisfies
$$
\|(\mathcal R^N)'v \|< C (0.1)^N  \|v\|
$$
where $C$ depends on $\eps, \alpha$ only.

\end{proof}
This implies  uniform contraction on $V_0$ for any $\alpha\in \mathcal B$.

%% file: stable.tex
\section{Stable foliation of $\cR$}

\subsection{Maps conjugated to the rotation lie in the strong stable foliation of $\cR$}

Let us begin with proving the macroscopic version of Theorem~\ref{th:infcont}:
\begin{theorem}
  \label{th:contraction}
Fix $\eps, C, \nu, \mu$.  Suppose that $f\in \mathcal D_{\eps}$,  $\rot (f) = \alpha\in \mathcal B$, $\Phi(\alpha)< C$, the map  is analytically conjugate to the rotation $R_\alpha$,
 $$
 f=\xi R_\alpha \xi^{-1},
 $$
where $\xi(0)=0$, $\xi$ is defined in $\Pi_\nu$, and $\xi(\Pi_{\nu})$ contains $B_{\mu}(0)$.

If $f$ is $\theta(\mu, \nu, C)$-close to $R_{\alpha}$, then $\mathcal R^n f$ are defined and $$\dist (\mathcal R^n f, \mathcal R^n R_{\alpha})\to 0\text{  as }n\to +\infty$$ at a geometric rate.

\end{theorem}

\begin{proof}
  If $f$ is $\theta=\theta(\mu, C)$-close to rotation, then the union of images of $B_{\mu}(0)$ under several first iterates of $f$ contains the strip $\Pi_{\mu/2}$. Thus
  $$\xi(\Pi_{\nu})\supset \Pi_{\mu/2}.$$

Since $f_\zeta$ is conjugate to $R_\alpha$ via $\xi$, then  $\mathcal R f_\zeta$ is conjugate to the rotation $R_\beta=\mathcal R R_\alpha$ via $\hat \xi = \Psi \xi (lz)$.
Informally, since $\Psi(z)\approx z/l$, the map $\xi\mapsto \hat \xi$ is a rescaling, thus contracts to $\text{Id}$.  This implies the statement.

More specifically, we start with making several renormalizations $$\mathcal R \colon \mathcal D_\eps \mapsto \mathcal D_{1.5\eps},\; \mathcal R \colon \mathcal D_{1.5\eps} \mapsto \mathcal D_{1.5^2\eps},\cdots$$ which are well-defined due to Lemma \ref{lem-domain2}. Since all the maps $\Psi_k$ (the change of coordinates $\Psi$ for the $k$-th iterate of $\cR$) used in these renormalizations are close to linear, after $N$ renormalizations, the map $\mathcal R^N f$ will be conjugate to a rotation $\mathcal R^N R_\alpha$ by a map $\tilde \xi$ that is defined in  $\Pi_{1.5^N \nu} $ and satisfies  $\tilde \xi(\Pi_{1.5^N \nu}) \supset \Pi_{1.5^N\mu/2}$. If
$f$ is sufficiently close to $R_\alpha$ and $N$ is sufficiently large, this implies that   $\tilde \xi'$ is close to $1$ in $\Pi_{\eps}$.

Thus we have reduced the general case to the case when $\eps$ is larger than a given universal constant and  $$\Pi_{\eps/3} \subset \xi(\Pi_{\eps/2}) \subset \Pi_{\eps}\text{  and }0.5<|\xi'|<2\text{ in }\Pi_{\eps/2}.$$

Note that the number of renormalizations $N$ we have to make only depends on $\mu, \nu$. The distance between $f$ and rotations required to have almost linear $\Psi_k$ only depends on $N$ and $\Phi(\alpha)$, due to the item \eqref{it-lin-ref} of the Technical Lemma and the lower bound on the lengths of fundamental intervals. Thus the required distance between $f$ and rotations only depends on $\mu, \nu, C$.

Further, we show that  $|\xi''/\xi'|_{C(\Pi_{\eps/2})}$ (nonlinearity of $\xi$) decreases provided that  $\xi $ is defined in a thick strip as above (thus close to identity) and $\eps$ is sufficiently large.

The direct computation shows that
\begin{equation}
 \label{eq-nonlin-xi}
 \frac{\hat \xi''}{\hat \xi'} = \left. \frac{\Psi''}{\Psi'}\right|_{\xi(lz)} \cdot \xi'(lz) \cdot l + l\left. \frac{\xi''}{\xi'}\right|_{lz}.
\end{equation}
Let us estimate the first summand. In the estimates below, let $c$ denote any universal constant (not necessarily the same constant as in the Technical Lemma).
For $\eps>1$, due to the Technical Lemma (item \ref{it-lin-ref}),
\begin{equation}\label{eq-Psi-lin}
 \|\Psi(zL) - z\|_{C(R)} \le c\eps^{1-2/p} \dist_{C(\Pi_{0.1\eps})}(f^n, R_{n\alpha}).
 \end{equation}
 The following lemma relates the latter distance to the nonlinearity of $\xi$.

  \begin{lemma}
  If $f$ is analytically conjugate to a rotation $R_\alpha$ via $\xi\colon \Pi_{\eps/2}\to \Pi_\eps$, and $\xi^{-1}$ is defined in $\Pi_{0.1\eps}$, then
$$  \dist_{C(\Pi_{0.1\eps})}(f^n, R_{n\alpha}) \le  \{n\alpha\} \cdot c \|\xi''/\xi'\|_{C(\Pi_{0.5\eps})}.$$
 \end{lemma}
  \begin{proof}
Since $\xi$ commutes with the unit translation, the quantities $|\xi'-1|$, $|(\xi^{-1})'-1|$ are bounded on each circle $\bbR/\bbZ+i\theta$ by $c\|\xi''/\xi'\|$, where $c$ in a universal constant. So for  $z\in \Pi_{0.1\eps}$, we have
\begin{eqnarray*}
  |f^n(z) - z - \{n\alpha\}| &=  |\xi(\xi^{-1}(z)+\{n\alpha\}) - \xi(\xi^{-1}(z + \{n\alpha\}))|\\
  & = |\xi'(\tau)| \cdot |\xi^{-1}(z)+\{n\alpha\} - \xi^{-1}(z + \{n\alpha\})|\\
  &= |\xi'(\tau)| \cdot |-(\xi^{-1})'(\lambda) \{n\alpha\}+\{n\alpha\}| \\
  & \le \{n\alpha\} \cdot c \|\xi''/\xi'\|_{C(\Pi_{0.5\eps})}.
\end{eqnarray*}
\end{proof}

Now it remains to estimate $\Psi''/\Psi'$ using our estimate \eqref{eq-Psi-lin} on $\Psi-z/l$. This part of  the argument resembles the end of the proof of item \eqref{it-deriv} in the Technical Lemma.

Extend $\Psi$ to the domain  $\hat R = [-\eps l, \eps l] \times [-1.5i\eps l, 1.5i\eps l]$ using iterates of $f^n$. We need approximately $\eps$ forward and backward  iterates. Since each iterate $f^{nk}$ is conjugate to $R_{nk\alpha}$ via $\xi$,  in the new rectangle, we estimate using the previous lemma and \eqref{eq-Psi-lin}
$$\dist_{\hat R} (\Psi, z\mapsto z/l)< (c\eps^{1-2/p}+c\eps)  \|\xi''/\xi'\|_{C(\Pi_{0.5\eps})}.$$
Now, Cauchy estimates in $$\tilde R=R\cap \Pi_{l\eps} \subset \hat R = [-\eps l, \eps l] \times [-1.5i\eps l, 1.5i\eps l]$$ imply that
 $$\|\Psi''\|_{C(\tilde R)} \le (c\eps^{1-2/p}+c\eps)l^{-2}\eps^{-2}  \|\xi''/\xi'\|_{C(\Pi_{0.5\eps})},$$
 $$\|\Psi' -1/l\|_{C(\tilde R)} \le (c\eps^{1-2/p}+c\eps)l^{-1}\eps^{-1}\, \|\xi''/\xi'\|_{C(\Pi_{0.5\eps})}.$$

Thus $l\cdot \Psi''/\Psi'$ in $\tilde R$ can be estimated by $c\eps^{-1}   \|\xi''/\xi'\|_{C(\Pi_{0.5\eps})}. $

As mentioned above, we assume that $$0.5<|\xi'(lz)|<2.$$

Finally, using \eqref{eq-nonlin-xi} we get
 $$
\left\|\frac{\hat \xi''}{\hat \xi'}\right\|_{C(\Pi_{\eps/2})}\le  0.1\left\|\frac{\xi''}{\xi'}\right\|_{C(\Pi_{\eps/2})}
$$
 if $\eps$ is larger than a universal constant.
Now Lemma \ref{lem-domain2} implies that all iterates of $\mathcal R$ are defined on $f$. The above estimate shows that $\dist (\mathcal R^n f, \mathcal R^n R_{\alpha})\to 0$  as $n\to +\infty$ at a geometric rate.
\end{proof}

\subsection{Strong stable manifolds of rotations $R_\alpha$ with $\alpha\in\cB$ consist of conjugacy classes}
Let $W^s(R_{\alpha})\subset \mathcal D_\eps$ be the stable manifold of $R_\alpha$ under the action of $\mathcal R$.
The next theorem holds the key to linking Risler's and Yoccoz's theorems. Let $C_0$ be the same universal constant as in Yoccoz's theorem.
\begin{theorem}
  \label{th:ourconj}
Suppose that $\alpha$ is a Brjuno number. Let $\cV\ni R_\alpha$ be a connected analytic submanifold in $\mathcal D_{\eps}$ such that $\cV\subset W^s(R_\alpha)$. If $f\in \cV$ is sufficiently close to $R_\alpha$,  then $f$ is analytically conjugate to $R_\alpha$ by a conformal change of coordinate $\xi $, $f = \xi R_\alpha \xi^{-1}$, in some open substrip of $\Pi_\eps$.

 
\end{theorem}
\begin{proof} 
  Let $V=\{k\alpha\mid k\in \bbZ\}\subset \bbR/\bbZ$.
  Consider  a map $\mathbf P$ on $W^s(R_\alpha)\times V$ defined by  $\mathbf P(f, k\alpha): = f^k(0)$.
  This map has the following properties:

\begin{itemize}
 \item $\mathbf P$ is defined on $ W^s(R_\alpha)\times V$.
 \item $\mathbf P$ is identical on $\{R_\alpha\}\times V$.
 \item $\mathbf P(\cdot, a)$ is holomorphic on $f$ when $a\in V$ is fixed.
 \item $\mathbf P(f, \cdot)$ is injective for each fixed $f$.
\end{itemize}
The last property holds true because otherwise $f^k(0)=f^m(0)$ for some $k,m,$ and since $f$ is univalent in $\Pi_\eps$, we have that $f^{|k-m|}(0)=0$. This implies that all high renormalizations of $f$ have a fixed point at $0$ and cannot be close to  $\mathcal R^n(R_\alpha)$, which contradicts $f\in W^s(R_\alpha)$.

It thus follows that $\mathbf P$ is a holomorphic motion over $\cV$. By the $\lambda$-lemma of Man\~e-Sad-Sullivan \cite{MSS} (see \cite{BR} for the infinite-dimensional version),
it extends to the holomorphic motion of the closure of the orbit of $0$ by $f\in\cV$. Since for $R_\alpha$ this orbit is the circle, we see that 
$$\gamma_f:=\overline {\{f^n(0), n\in \bbZ\}}$$ is a quasicircle which is 
homotopic to $\bbR/\bbZ$ in $\Pi_\eps$. The map $f|_{\gamma_f}$ is quasiconformally conjugate to $R_\alpha|_\TT$ and thus has rotation number $\alpha$.

Take an open cylinder between $\gamma_f$ and the upper boundary of $\Pi_\eps$. Construct a holomorphic mapping  $\Xi_+$ to a standard cylinder $\{z\in \bbC/\bbZ \mid 0<\Im z<a\}$ with $\Xi_+(0)=0$. Extend $\Xi_+$ to the boundaries by the Carath\'eodory principle. Now $f$ is conjugate via $\Xi_+$ to a map $g_+=\Xi_+ f \Xi_+^{-1}$ that is analytic above $\bbR/\bbZ$, continuous up to $\bbR/\bbZ$, and preserves $\bbR/\bbZ$. Due to the Schwarz Reflection Principle, we can  extend $g_+$ to a full neighborhood of $\bbR/\bbZ$, hence  $g_+$ is analytic in this neighborhood.

We now apply Yoccoz's Theorem~\ref{th-Yoc1} to conjugate $g_+$ to $R_\alpha$:
$$h_+\circ g_+ \circ h^{-1}_+=R_\alpha,$$
where $h_+$ is conformal in some strip around $\bbR/\bbZ$,  and $h_+(0)=0$.

Let us now repeat the above procedure in the open cylinder between $\gamma$ and the {\it lower} boundary of $\Pi_\eps$. In the same way we obtain maps $\Xi_-$ and $h_-$.
We see that $$\phi_\pm\equiv h_\pm\circ \Xi_\pm$$ conjugate $f$ to $R_\alpha$ on their respective domains of definition. Moreover, they  extend homeomorphically to the invariant quasicircle $\gamma_f$, mapping it to $\TT$. Since both of them map $0$ to $0$, they coincide on $\gamma_f$.
As quasicircles are conformally removable, $\phi_\pm$ glue along $\gamma_f$ to an analytic map defined in a strip around $\gamma_f$.


%
 \end{proof}
\subsection{Thickness of the band}
We will prove that the conjugacy from Theorem \ref{th:ourconj} is defined in a uniformly thick strip.
    We need a uniform estimate on $\dist(\gamma_f, \bbR/\bbZ)$ depending on $\dist(f, R_\alpha)$, where $\gamma_f = \overline {\{f^n(0)\}}$. Recall that
    $$\mathcal B_{C} = \{\alpha \in \mathcal B \mid \Phi(\alpha)<C\}.$$
    \begin{lemma}
   \label{lem-curves}
        For $M$, $\eps$ fixed, for each $a$ there exists $b$ with the following property.

        Suppose that $f\in \mathcal D_{\eps}$ is an analytic map such that $f|_{\gamma_{f}}$ has rotation number $\alpha, \Phi(\alpha)<C$, and $f$ lies in the stable manifold $W^s(R_\alpha)$.

        Then $|f-R_\alpha|<b$ implies that  $\gamma_{f}$ parametrized by $\xi_f$ is $a$-close to $\bbR/\bbZ$ in $C$ metric.
    \end{lemma}
    \begin{proof}

Let $\Psi_1, \Psi_2, \dots$ be the charts that correspond to $f$ and its successive renormalizations. For each $n$, consider a small rectangle $R_n=\Psi^{-1}_1\dots \Psi^{-1}_n ([0,1]\times [-\eps, \eps])$. Since $f$ is infinitely renormalizable, the finite number $N(n)$ of images of $R_n$ under $f$ forms a strip $S_n$ in $\Pi_{\eps}$ that contains the orbit of zero.

Choose $n$ such that for any rigid rotation $R_\alpha$ with $\Phi(\alpha)<C$, we have $\mathrm{diam}\, R_n< a/2$. Then for sufficiently small $b$, we have the same inequality for $f$ whenever $|f-R_\alpha|<b$, $\rho( f )= \alpha, \Phi(\alpha)<C$, because $\Psi_k$, $k\le n$, are close to linear  (see item \eqref{it-lin-ref} of the Technical Lemma) and the lengths of the fundamental intervals used in the first $n$ renormalizations are bounded from below whenever $\Phi(\alpha)<C$. Also, for  sufficiently small $b$, all $(f^k)'$, $k\le N(n)$, are close to $1$ in $\Pi_{\eps/2}$ whenever $|f-R_\alpha|<b$. So each arc $f^{k} (R_n) \cap \gamma_f$ of $\gamma_f$  belongs to the $a$-neighborhood of the corresponding arc of $\bbR/\bbZ$ for $R_\alpha$. The claim follows.

    \end{proof}
This implies a stronger version of Theorem \ref{th:ourconj} above. Let $$\hat \eps : = \eps - \frac {1}{2\pi} C -C_0$$ where $C_0$ is  a constant from Yoccoz's theorem \ref{th-Yoc2}. Let $\mathcal U_\kappa(g)$ denote the $\kappa$-negihborhood of $g$ in $\mathcal D_\eps$.

\begin{lemma}
\label{lem-charts}
 In the assumptions of Theorem  \ref{th:ourconj} above, for $C$ fixed and for any  $\eps$ such that $\hat \eps > 0.1$,  there exists $\kappa>0$ such that if $\Phi(\alpha)<C$ and $f\in \mathcal V \cap \mathcal U_\kappa (R_\alpha)\subset \mathcal D_\eps$,
 then $\xi_f$ is defined on $\Pi_{\hat \eps-0.1}$.

\end{lemma}
\begin{proof}
Note that our assumptions imply $\rot (f|_{\gamma_f}) =\alpha$ because $\mathcal V$ is connected.
  Using the previous lemma, we may, for each $a$, choose $\kappa$ so that $\gamma_f$ is $a$-close to $\bbR/\bbZ$ whenever $f\in \mathcal V \cap \mathcal U_\kappa(R_\alpha)$.

In the proof of Theorem  \ref{th:ourconj}, the chart $\xi_f$ was constructed as a composition of uniformizing maps $\Xi_\pm$ and maps $h_{\pm}$ provided by Yoccoz's theorem.

  For every $v>0$ there exists $a>0$ such that assuming that $\gamma_f$ is $a$-close to $\TT$ then $\Xi_\pm$ is $v$-close to the identity in $0.1\eps< |\Im z|< \eps$ (see e.g. Lemma~2.1 of \cite{BBY3}). This implies that if $a$ is small, then $g_{\pm} = \Xi_{\pm} f \Xi_{\pm}^{-1}$ in the proof of Theorem  \ref{th:ourconj} is uniformly close to identity in $0.2\eps< |\Im z|< \eps$, and thus in $\Pi_\eps$ by the Maximum Principle.

  Recall that $h_{\pm}$ is the analytic conjugacy of a circle diffeomorphism $g_{\pm} $ to a rotation. Due to  Yoccoz's theorem \ref{th-Yoc2}, $h_{\pm}$ are defined in  $\Pi_{\hat \eps}$ for  $\hat \eps > 0$.





  Since $\Xi_{\pm}$ is close to identity in $0.1\eps< |\Im z|< \eps$, this implies that $\xi_f$ is defined in
  $\Pi_{\hat \eps-0.1}$ as required.


\end{proof}

\subsection{Leafs of the stable foliation of $\cR$}
To complete our discussion of the renormalization picture, it remains to construct the stable foliation of $\cR$.

\begin{theorem}
  \label{th:leafs}
  For every $C$, for each $\eps>c_1 C+c_2$ where $c_1, c_2$ are universal constants, there exists $\kappa>0$ such that the following holds. For every $\alpha, \Phi(\alpha)<C$, consider $\mathcal R \colon \mathcal D_\eps \to \mathcal D_\eps$ and take the connected component $\cV_\alpha$ of  the intersection
  $W^s(R_\alpha)\cap \mathcal U_\kappa(R_\alpha)$. Then $\cV_\alpha$  is a  complex analytic submanifold of $\cD_\eps$ of codimension one. Moreover, $\cV_\alpha$ is an analytic graph over a hyperplane.
\end{theorem}

\begin{proof}
  Similar statements exist in the literature, as general corollaries of uniform hyperbolicity of
  compact operators in infinite dimensions; see, for instance, Lemma~VII.1 of \cite{Mane1981}.
  However, we were not able to find a
general  theorem which would fit our needs, so we will supply a proof starting from the stable/unstable manifold theory for periodic orbits of $\cR$.

  Let us write $\alpha=[a_0,a_1,\ldots]$ and set $\alpha_n$ equal to the continued fraction obtained by repeating $a_0,a_1,\ldots,a_{n-1}$ periodically.
 There exists $n_0\in\NN$ such that for all $n\geq n_0$ we have
  \begin{equation}
    \label{eq:phibound}
|\Phi(\alpha_n)|<2|\Phi(\alpha)|+c
    \end{equation}
    with a universal $c$; this is evident for $\Phi_0$ and thus holds for $\Phi$, the universal constant taking care of the difference between the functions.
  For $\eps$ sufficiently large, $R_{\alpha_n}$ is a hyperbolic periodic point of $\cR$ due to Theorem \ref{th:infcont}. So it has an analytic local stable submanifold $$\cV_n:= W^s_\text{loc}(R_{\alpha_n})$$
  of codimension one (see e.g. \cite{Vand,ElBi} for an infinite-dimensional version of Hadamard-Perron Theorem).
  Now we will construct $\cV_\alpha$ as a limit of $\cV_n$.
    We will prove that submanifolds $\cV_n$ are graphs over one and the same open domain in $V_0$; extracting a converging subsequence for $\alpha_n\to \alpha$, we will find a stable submanifold at $\alpha$ as a limit of $\cV_n$.


\begin{lemma}
\label{lem-graph}
For any $C$, for each  $\eps>c_1 C+c_2$ where $c_1, c_2$ are universal constants,  there exists $\kappa$ such that for any $\alpha, \Phi(\alpha)<C$,  the intersection of any $\cV_n$ with $\kappa$-neighborhood of rotations $U_{\kappa}(\mathcal T)$ is a graph over some subdomain of $V_0$ whenever defined.
\end{lemma}
\begin{proof}
Recall that $\cV_n$ is a codimension-1 analytic submanifold in $\mathcal D_\eps$. Due to Theorem   \ref{th:ourconj}, $\cV_n$ consists of maps that are analytically conjugate to $R_{\alpha_n}$ by conjugacies $\xi_f$.

As above, for a map $f\in\cV_n$, let us denote by $\gamma_f$ its invariant analytic circle containing $0$, and let $\xi_f$ be the analytic linearization with $\xi_f(0)=0$. 

Due to Lemma \ref{lem-charts}, for sufficiently large $\eps$ and small $\kappa$,   whenever $f\in \mathcal U_\kappa(\mathcal T)$,  $\xi_f$ is defined on $\Pi_{\hat \eps-0.1}$ and $ \xi_f(\Pi_{\hat \eps-0.1}) \subset \Pi_\eps$. Here  $\hat \eps$ is determined as above, for $\alpha_n$ instead of $\alpha$. Due to \eqref{eq:phibound}, for sufficiently large $\eps$, this implies $|\xi_f'-1|<0.1$ on $\bbR/\bbZ$.

Recall that the tangent space to $\cV_n$ at $R_\alpha$ coincides with  $$\{v\in T\mathcal D_\eps \mid \int_{\bbR/\bbZ} v dz =0\}.$$ Similarly, the tangent space to $\cV_n$ at $f$ may only contain vector fields $$\{v\in T\mathcal D_\eps \mid \int_{\bbR/\bbZ} v d(\xi_f^{-1}) =0\};$$ since it must be codimension 1, it coincides with this space. Note that this tangent space never intersects the open cone $$\{v \in T\mathcal D_\eps \mid \sup_{\Pi_{\eps}}|v(z)  - v(0)| < 0.1 |v(0)| \}$$ around the complex space generated by a vector field $d/dz$; this follows from the estimate on $\xi_f'$.  Thus $\cV_n$ projects to the space $V_0$ along the vector field $d/dz$.



\end{proof}

  Consider a point $f$ on a relative boundary of $\cV_n$ (that is, $f$ is in a closure of $\cV_n$, but $\cV_n$ is not a local submanifold at $f$). We will prove that it cannot belong to $\cU_\kappa(\mathcal T)$. This will imply the statement of the theorem.

  \begin{lemma}
  Suppose that for $C$ fixed,  $\eps>c_1 C +c_2$ where $c_1, c_2$ are universal constants, and $\kappa$ is sufficiently small.
   Then for any $\alpha, \Phi(\alpha)<C$, the relative boundary of $\cV_n$ does not intersect $\cU_\kappa(\mathcal T)$ (i.e. it belongs to its boundary).
  \end{lemma}
  \begin{proof}

    Assume the contrary.  Take $f$ on the relative boundary of $\cV_n$ that belongs to $\mathcal U_\kappa(\mathcal T)$. Take a sequence $g_k \to f$, $g_k \in \cV_n$. Due to Theorem \ref{th:ourconj}, all $g_k$ are analytically conjugate to $R_{\alpha_n}$ via maps $\xi_k$; let $\gamma_k = \xi_k(\bbR/\bbZ)$.

    Recall that $\xi_k$ are defined in $\Pi_{\hat\eps-0.1}$ due to Lemma \ref{lem-charts}.

    This enables us to extract a convergent subsequence from $\xi_k^{-1}$ and conclude that $f$ is analytically conjugate to $R_\alpha$  via $\xi$ such that $\xi$ is defined in $\Pi_{\hat \eps-0.1}$.


    If $\kappa $ was chosen sufficiently small to satisfy assumptions of  Theorem \ref{th:contraction}, then we get a contradiction, because $\mathcal R^{nk} f$ then converge to $\mathcal R^{nk} R_{\alpha_n} = R_{\alpha_n}$ as $k\to \infty$ geometrically fast, thus $f\in W^s(R_{\alpha_n})$, and is not on a boundary.
%
  \end{proof}

Hence all points on the relative boundaries of $\mathcal V_n$ do not belong to $U_{\kappa}(\mathcal T)$.  Hence all  $\mathcal V_n$ are graphs over the same open ball $B_{\kappa}$ in $V_0$. Extracting a convergent subsequence, we find an analytic submanifold $\cV_\alpha\ni R_{\alpha}$ that is a limit of $\cV_n$, and is an analytic graph over $B_\kappa$.

For each $f\in \cV_\alpha\cap U_{\kappa}(\mathcal T)$, there exists a sequence $f_n\in \cV_n, f_n\to f$.
Same arguments as in the previous lemma show that $f$  satisfies assumptions of Theorem~\ref{th:contraction}, and thus lies in $W^s(R_\alpha)$.


This proves Theorem   \ref{th:leafs}.

  \end{proof}

\subsection{Conjugacies in thin strips}

Here we will prove that if a map is close to $R_\alpha$ and analytically conjugate to it in a thin strip, then it  belongs to $\mathcal V_\alpha$ (thus is conjugate to $R_\alpha$ in a thick strip). This is based on Theorem \ref{th:leafs} and  strengthens Theorem \ref{th:contraction} above.

\begin{lemma}
\label{lem-conj}
In assumptions of Theorem \ref{th:leafs},  let $f$ be  analytically conjugate to  $R_\alpha$ in some open substrip of $\Pi_{\eps/3}$ and sufficiently close to $R_{\alpha}$. Then $f\in \mathcal V_\alpha$.
\end{lemma}
  The proof is contained in \cite[p.91-92]{Risler}, but we include it for the sake of completeness.
\begin{proof}
 Recall that Theorem \ref{th:leafs} shows that $\mathcal V_\alpha$ is an analytic graph over a $\kappa$-ball in the hyperplane $V_0$ where $\kappa$ only depends on $\Phi(\alpha)$ and $\eps$. Thus if $f$ is $\kappa$-close to $R_\alpha$, then  there exists $g=f+\lambda \in  \mathcal V_\alpha$, $\lambda\neq 0$. Consider the family $f+t$ in the chart $\xi_g = \xi_{f+\lambda}$: take
$$G_t := \xi_g \circ (f+t)\circ \xi_g^{-1},$$ then we have $G_\lambda=R_\alpha$, and $G_0$ is conjugate to $f$.

However, due to Lemma \ref{lem-charts}, $\xi_g$ is defined in $\Pi_{\hat \eps-0.1}$; for sufficiently large $\eps$, this implies $$|\xi_g' -1|<0.05\text{ in }\Pi_{\eps/3},$$ thus $$\left|\frac {\partial}{\partial t} G_t -1\right| <0.1\text{ in }\Pi_{\eps/3}.$$ Since $G_\lambda =R_\alpha$, we have $$\dist ( G_0, R_{\alpha-\lambda}) < 0.1 |\lambda| \text{ in }\Pi_{\eps/3}.$$ Clearly, this implies that $G_0$ cannot be conjugate to $R_\alpha$ in any substrip of $\Pi_{\eps/3}$.
The contradiction shows that $f\in \mathcal V_\alpha$.
\end{proof}

%% file: conclusion.tex
\section{Proof of the Main Theorem}
Items 1 and 2 of the Main Theorem (renormalization operator as a real-symmetric complex-analytic operator near $\mathcal T$ with compact derivative) follow from Lemma \ref{lem-domain2}.

Unstable direction is described in Lemma \ref{lem-expand}.
Theorem \ref{th:leafs} shows that for $\eps>c_1 C +c_2$ where $c_1, c_2$ are universal constants, the local stable manifold $\mathcal V_\alpha$ at $R_\alpha, \alpha\in \mathcal B_C$, is an analytic  graph over a hyperplane.   Theorem \ref{th:ourconj} together with Lemma \ref{lem-charts}  shows that a germ of $\mathcal V_\alpha$ only contains diffeomorphisms that are analytically conjugate to $R_\alpha$, and the conjugacy is defined on  $\Pi_{\tilde\eps-0.1}$, which contains $\Pi_{0.4\eps}$ for large $\eps$.
Due to Lemma \ref{lem-conj}, if $f$ is analytically conjugate to $R_\alpha$ on a substrip of $\Pi_{\eps/3}$ and sufficiently close to $R_\alpha$, then $f \in \mathcal V_\alpha$.

This concludes the proof of the Main Theorem.
\section{Risler's theorem}
\label{sec-conclusion}
Let us explain how Theorem \ref{th-submanif}  (and thus also Risler's theorem) follows from the Main Theorem.

 Fix $\eps$ and a sequence $\{s_k\}$. Using Proposition \ref{prop-Yoc}, find  $m=m(\eps, \{s_k\})$ such that for any $\alpha\in  \mathcal B_{\{s_k\}}$, we have $\tilde \eps = \eps \cdot \{q_m\alpha\} \gg c_1 \Phi(\alpha_m)+c_2$ where $c_1, c_2$ are as in the statement of the Main Theorem.


Moreover, the Proposition provides us with $\kappa'$ such that if $\rot (f) = \alpha\in\mathcal B_{\{s_k\}}$ and $f$ is $\kappa'$-close to rotations, then $\mathcal R_{\tilde\eps, q_m} f$ is $\kappa$-close to rotations in $\mathcal D_{\tilde \eps -\eps'}$ where $\eps'$ is small and $\kappa$ is as in the Main Theorem.



Let an analytic submanifold $\mathcal W_\alpha\subset \mathcal D_\eps$ at $R_\alpha$ be formed by maps such that $\mathcal R_{\tilde\eps, q_m} f \in \mathcal V_{\alpha_m} \subset \mathcal D_{\tilde \eps-\eps'}$. Due to Lemma \ref{lem-charts}, all  $g \in \mathcal V_{\alpha_m}$ that are $\kappa$-close to $R_{\alpha_m}$ are analytically conjugate to $R_{\alpha_m}$ in the strip of width $$\tilde \eps-\eps'-\Phi(\alpha_m)-C_0-0.1 \gg 0.7 \tilde \eps;$$ for sufficiently large $\tilde \eps$, this implies
$$\|\xi_g'-1\|_{C(\Pi_{0.7\tilde  \eps})}<0.1.$$ Therefore, $f\in \mathcal W_\alpha\cap \mathcal U_{\kappa'}(R_\alpha)$ implies that $f$ is analytically conjugate to $R_\alpha$ in the strip formed by images of $\Psi^{-1}(\Pi_{0.7\tilde \eps})$ under $f$. Since the chart $\Psi$  is close to linear for $f$ close to rotations, this strip contains $\Pi_{0.6\eps}$, and the conjugacy $\xi_f$ satisfies $$\|\xi_f'-1\|_{C(\Pi_{0.6\eps})}<0.2.$$ So $f\in \mathcal W_\alpha\cap \mathcal U_{\kappa'}(\mathcal T)$ implies that the conjugacy is in $\cD_{0.6\eps}$.

The converse is proved in the same way as in Lemma \ref{lem-conj}. This concludes the proof of Theorem \ref{th-submanif}.

Now let us  prove the real-symmetric version, Corollary~\ref{cor-submanif1}. Clearly,  the set of circle diffeomorphisms with fixed rotation number $\alpha\in \mathcal B$ coincides with the  graph of a continuous map $$f\mapsto a: \rot(f+a)=\alpha .$$ The Main Theorem shows that the germ of an  analytic submanifold $\mathcal V_\alpha^{\bbR}$ at $R_\alpha$ is formed solely by circle diffeomorphisms.  Hence this germ coincides with the germ of the set of diffeomorphisms with rotation number $\alpha$. 

As for Corollary~\ref{cor-tongue}, Yoccoz's Theorem~\ref{th-Yoc1} implies that there exists $\delta_{\{s_k\}}$ such that if $\alpha\in\cB_{\{s_k\}}$ and $a\in[0,\delta_{\{s_k\}})$, for every $(\mu,a)\in\mathbb A_\alpha$ the analytic diffeomorphism of the circle  $F_{\mu,a}$ is conformally conjugate to the rigid rotation $R_\alpha$. Theorem~\ref{th:contraction} implies that $$\{F_{\mu,a}\;|\;(\mu,a)\in\mathbb A_\alpha\}= W^s(R_\alpha)\cap \{F_{\mu,a}\;|\;\mu\in[0,1),\;a\in[0,\delta_{\{s_k\}})\},$$
and the claim readily follows.
Corollary~\ref{cor-disk} is similarly straightforward.